\documentclass[12pt]{amsart}  
\usepackage{amssymb}
\usepackage{amscd}
\usepackage{multicol}

\newtheorem{theorem}{Theorem}[section]
\newtheorem{corollary}[theorem]{Corollary}
\newtheorem{proposition}[theorem]{Proposition}
\newtheorem{lemma}[theorem]{Lemma}

\theoremstyle{remark}
\newtheorem{remark}[theorem]{Remark}
\newtheorem{example}[theorem]{Example}
\newtheorem*{notation}{Notation}

\theoremstyle{definition}
\newtheorem{definition}[theorem]{Definition}

\numberwithin{equation}{subsection}

\DeclareMathOperator{\Symm}{\mathtt{S}}

\DeclareMathOperator{\Map}{{\mathcal{H}\!\mathit{om}}}
\DeclareMathOperator{\Hom}{Hom}

\DeclareMathOperator{\shEnd}{\underline{End}}

\DeclareMathOperator{\im}{im}
\DeclareMathOperator{\coker}{coker}

\newcommand{\pr}{\mathtt{pr}}
\newcommand{\ev}{\mathtt{ev}}
\newcommand{\id}{\operatorname{id}}
\newcommand{\ip}{{\langle\ ,\ \rangle}}

\newcommand\OExt[1]{{\mathcal{O}[#1]\operatorname{Ext}}}

\newcommand\CExt[2]{\operatorname{\mathcal{C}Ext}({#1})_{#2}}

\newcommand{\DGMan}{\mathrm{dg}\mathfrak{Man}}
\DeclareMathOperator{\Cour}{\mathtt{Q}}
\DeclareMathOperator{\CA}{CA}
\DeclareMathOperator{\ECA}{ECA}
\DeclareMathOperator{\Conn}{Conn}

\newcommand\Mod[1]{{#1}\!-\!\mathtt{Mod}}
\newcommand\LA[1]{{#1}\!-\!\mathtt{LieAlgd}}

\begin{document}

\title{On higher-dimensional Courant algebroids}

\author{Paul Bressler}
\address{Universidad de los Andes}
\email{paul.bressler@gmail.com}
\author{Camilo Rengifo}
\address{Universidad de La Sabana}
\email{camilo.rengifo@unisabana.edu.co}

\begin{abstract}
We define the transgression functor which associates to a (higher-dimensional) Courant algebroid on a manifold a Lie algebroid on the shifted tangent bundle of the manifold.
\end{abstract}

\maketitle

\section{Introduction}
Courant algebroids (\cite{C}, \cite{D}, \cite{LWX}, \cite{B},
\cite{S1}, \cite{S3}), as well as higher-dimensional
generalizations thereof (\cite{Z}, \cite{V2}, \cite{CR}) are
objects of finite dimensional differential geometry built around
Leibniz (rather than Lie) brackets which give rise to
field-theoretic Poisson brackets. Leibniz brackets arise as
derived brackets (\cite{YKS1}, \cite{YKS2}) and it is natural to
expect Courant algebroids to arise in this way. Such a
construction was given by D.~Roytenberg in \cite{R} and
P.~\v{S}evera in \cite{S3}, Letter 7; (see also \cite{K-S},
\cite{U}, \cite{Z}).

We describe a construction, which we refer to as \emph{transgression}, which associates to a Courant algebroid $\mathcal{Q}$ on a manifold $X$ (see Definition \ref{defn: courant}) a Lie algebroid $\tau\mathcal{Q}$ on the differential graded manifold $X^\sharp$ (frequently denoted by $T[1]X$ in the literature) with the underlying space $X$ and the structure sheaf the de Rham complex of $X$. The Lie algebroid $\tau\mathcal{Q}$ possesses (and, hence, is uniquely determined by) a universal property with respect to suitably defined Courant-to-Lie morphisms (see Definition \ref{def: CtL morphism}) of $\mathcal{Q}$ to Lie algebroids on $X^\sharp$. The construction is given in Section \ref{section: transgression for courant algebroids} which is devoted to the proof of the main result of the paper (Theorem \ref{thm: Q/CtL has initial object} and Corollary \ref{cor: tau section}). Furthermore, in Theorem \ref{thm: ext to ext equiv} we show that for a special class of Courant algebroids (Courant extensions of Lie algebroids, see Definition \ref{def: courant extension}) transgression restricts to an equivalence of appropriately defined categories. The Courant algebroid $\mathcal{Q}$ appears as a graded component of $\tau\mathcal{Q}$ and is equipped with the structure (the Leibniz bracket in particular) derived from the Lie algebroid structure of the latter.

The present work generalizes and expands upon the results of \cite{CR} which treated the case of exact (higher dimensional) Courant algebroids.

In Section \ref{section: the odd path space} we review some basics
of the theory of differential-graded manifolds; (see also
\cite{R}, \cite{S1}, \cite{V1}). In particular, we show that, for
$X$ a manifold, the differential-graded manifold $X^\sharp$ is a
mapping space, a well-known fact, \cite{S2}.

In Section \ref{section: transgression for lie algebroids} we define the transgression functors which associate to Lie algebroids and modules over these on $X$ corresponding objects on $X^\sharp$. In addition we introduce \emph{marked Lie algebroids} examples of which are the Lie algebroids $\tau\mathcal{Q}$ arising from Courant algebroids via transgression.

In Section \ref{section: courant algebroids} we give the
definition of (higher-dimensional) Courant algebroids and study
basic properties thereof, including a review of the classification
of exact Courant algebroids. Most of the results are direct
generalizations of the corresponding results of \cite{B}.

Section \ref{section: transgression for courant algebroids} contains the key construction
of the present paper, namely, of transgression for Courant algebroids,
as well as the principal results, Theorem \ref{thm: Q/CtL has initial object}, Corollary \ref{cor: tau section}.

In Section \ref{section: transgression for extensions} we specialize to the case of Courant extensions
and show (Theorem \ref{thm: ext to ext equiv}) that transgression restricts to an equivalence of appropriately defined categories.

For the convenience of the reader those proofs which consist of straightforward but lengthy verification of identities are delegated to Section \ref{section: proofs}.

\subsection{Notation}
In order to simplify notations in numerous signs we will write ``$a$" instead of ``$\deg(a)$" in expressions appearing in exponents of $-1$.
For example, $(-1)^{ab - 1}$ stands for $(-1)^{\deg(a) \cdot \deg(b) - 1}$.

Throughout the paper ``manifold" means a $C^\infty$, real analytic or complex manifold.
For a manifold $X$ we denote by $\mathcal{O}_X$ the corresponding structure sheaf of \emph{complex valued} $C^\infty$, respectively analytic
or holomorphic functions. We denote by $\mathcal{T}_X$ (respectively, by $\Omega_X^k$) the sheaf of complex valued vector fields
(respectively, differential forms of degree $k$) on $X$.

\section{The odd path space}\label{section: the odd path space}

\subsection{DG-manifolds}
For the purposes of the present note a \emph{differential-graded manifold (DG-manifold)} is a pair
$\mathfrak{X} := (X, \mathcal{O}_\mathfrak{X})$, where $X$ is a manifold and $\mathcal{O}_\mathfrak{X}$
is a sheaf of commutative differential-graded algebras (CDGA) on $X$ locally isomorphic to one of the form
$\mathcal{O}_X\otimes\Symm(E)$, where $\Symm(E)$ is the symmetric algebra of a finite-dimensional graded vector space $E$.

Let $\mathfrak{X} = (X, \mathcal{O}_\mathfrak{X})$ and
$\mathfrak{Y} = (Y, \mathcal{O}_\mathfrak{Y})$ be DG-manifolds. A
morphism $\phi\colon \mathfrak{X} \to \mathfrak{Y}$ is a morphism
of ringed spaces, which is to say a map $\phi\colon X \to Y$ of
manifolds together with the morphism of graded algebras
$\phi^*\colon \phi^{-1}\mathcal{O}_\mathfrak{Y} \to
\mathcal{O}_\mathfrak{X}$ compatible with the map
$\phi^{-1}\mathcal{O}_Y \to \mathcal{O}_X$.

We denote the category of DG-manifolds by $\DGMan$. Let $\DGMan^+$
denote the full subcategory of DG-manifolds $\mathfrak{X} = (X, \mathcal{O}_\mathfrak{X})$ such that $\mathcal{O}_\mathfrak{X}^i = 0$ if $i < N$ for some $N\in\mathbb{Z}$.

\begin{example}\label{example: shifted T}
An ordinary manifold is an example of a DG-manifold with the
structure sheaf concentrated in degree zero. Each ordinary
manifold $X$ determines a DG-manifold $X^\sharp \in \DGMan^+$
defined by $X^\sharp = (X,\Omega^\bullet_X, d)$ and frequently denoted by $T[1]X$ in the literature. There is a
canonical morphism $X \to X^\sharp$ of DG-manifolds defined by the
canonical map $\Omega^\bullet_X \to \mathcal{O}_X$.
\end{example}
\begin{example}\label{example: odd line}
Let $\vec{\mathfrak{t}}$ denote the DG-manifold with the
underlying space  consisting of one point and the DG-algebra of
functions $\mathcal{O}_{\vec{\mathfrak{t}}} =
\mathbb{C}[\epsilon]$, the free graded commutative algebra with
generator $\epsilon$ of degree $-1$ and the differential
$\partial_\epsilon \colon \epsilon \mapsto 1$. Note that
$\vec{\mathfrak{t}} \in \DGMan^+$.
\end{example}

\subsection{Products}
The category $\DGMan^+$ has finite products. Suppose that
$\mathfrak{X} = (X, \mathcal{O}_\mathfrak{X}), \mathfrak{Y} = (Y,
\mathcal{O}_\mathfrak{Y}) \in \DGMan^+$.

If $\mathcal{O}_\mathfrak{X} \cong \mathcal{O}_X\otimes\Symm(E)$
and $\mathcal{O}_\mathfrak{Y} \cong \mathcal{O}_X\otimes\Symm(F)$,
where $E$ and $F$ are finite-dimensional graded vector spaces,
then
\[
\mathcal{O}_{X\times Y}\otimes_{\mathcal{O}_X\boxtimes\mathcal{O}_Y} \mathcal{O}_\mathfrak{X}\boxtimes\mathcal{O}_\mathfrak{Y} \cong \mathcal{O}_{X\times Y}\otimes \Symm(E\oplus F)
\]
and is bounded below.

Let
\[
\mathcal{O}_{\mathfrak{X}\times\mathfrak{Y}} := \mathcal{O}_{X\times Y}\otimes_{\mathcal{O}_X\boxtimes\mathcal{O}_Y} \mathcal{O}_\mathfrak{X}\boxtimes\mathcal{O}_\mathfrak{Y} .
\]
Then, $(X\times Y, \mathcal{O}_{\mathfrak{X}\times\mathfrak{Y}})  \in \DGMan^+$.

The morphism $\pr_\mathfrak{X} \colon
\mathfrak{X}\times\mathfrak{Y} \to \mathfrak{X}$ is defined as the
projection $\pr_X \colon X\times Y \to X$ together with the map
$\pr_\mathfrak{X}^* \colon \pr_X^{-1}\mathcal{O}_\mathfrak{X} \to
\mathcal{O}_{\mathfrak{X}\times\mathfrak{Y}}$ defined as the
composition
\[
\pr_X^{-1}\mathcal{O}_\mathfrak{X} \xrightarrow{\id\boxtimes 1} \mathcal{O}_\mathfrak{X}\boxtimes\mathcal{O}_\mathfrak{Y} \xrightarrow{1\boxtimes \id} \mathcal{O}_{\mathfrak{X}\times\mathfrak{Y}} .
\]
The morphism $\pr_\mathfrak{Y} \colon \mathfrak{X}\times\mathfrak{Y} \to \mathfrak{Y}$ is defined similarly.

\begin{lemma}
The DG-manifold $(X\times Y,
\mathcal{O}_{\mathfrak{X}\times\mathfrak{Y}}) \in \DGMan^+$
together with the morphisms $\pr_\mathfrak{X}$ and
$\pr_\mathfrak{Y}$ represents the product of $\mathfrak{X}$ and
$\mathfrak{Y}$.
\end{lemma}
\begin{proof}
Left to the reader.
\end{proof}

\subsection{The odd path space}\label{subsection: the odd path space}
We briefly recall the basic facts concerning the internal hom in
presheaves. For a category $\mathcal{C}$ we denote by
$\widehat{\mathcal{C}}$ the category of presheaves on
$\mathcal{C}$. For $F, G \in \widehat{\mathcal{C}}$ the presheaf
$\Map(F,G)$ is defined by
\[
\Map(F,G)(c) = \Hom_{\widehat{\mathcal{C}}}(c,\Map(F,G)) := \Hom_{\widehat{\mathcal{C}}}(c\times F, G)
\]
for $c\in\mathcal{C}$. Then, for any $H \in \widehat{\mathcal{C}}$,
\[
\Hom_{\widehat{\mathcal{C}}}(H,\Map(F,G)) \cong \Hom_{\widehat{\mathcal{C}}}(H\times F, G)
\]
naturally in $F$, $G$ and $H$. The case $H=\Map(F,G)$ and
$\id_{\Map(F,G)}$ gives rise to the \emph{evaluation map}
$\ev\colon \Map(F,G)\times F \to G$. The case $H = G\times F$ and
the projection $\pr \colon G\times F \to G$ gives rise to the map
$i \colon G \to \Map(F,G)$. In other words, internal hom is
defined as the right adjoint to the product and the maps $\ev$ and
$i$ are the counit and the unit of the adjunction.

\begin{theorem}\label{thm: odd path space representable}
For a manifold $X$ the DG manifold $X^\sharp$ (see Example \ref{example: shifted T}) represents the presheaf $\Map(\vec{\mathfrak{t}}, X)$ on $\DGMan^+$.
\end{theorem}

Implicit in the statement of the theorem are natural bijections
\begin{equation}\label{Map to sharp}
\Map(\vec{\mathfrak{t}}, X)(\mathfrak{Y}) = \Hom(\mathfrak{Y}\times\vec{\mathfrak{t}}, X) \xrightarrow{\phi\mapsto\phi^\sharp} \Hom(\mathfrak{Y}, X^\sharp) .
\end{equation}
and
\begin{equation}\label{sharp to Map}
\Hom(\mathfrak{Y}, X^\sharp) \xrightarrow{\psi\mapsto\psi_\flat} \Hom(\mathfrak{Y}\times\vec{\mathfrak{t}}, X) = \Map(\vec{\mathfrak{t}}, X)(\mathfrak{Y}) .
\end{equation}
which we are going to define presently.

Let $\phi \colon \mathfrak{Y}\times \vec{\mathfrak{t}} \to X$ be a
morphism of DG manifolds. The morphism $\phi$ is given by a map of
manifolds $\phi \colon Y \to X$ and a morphism of DGA $\phi^*
\colon \phi^{-1}\mathcal{O}_X \to
\mathcal{O}_\mathfrak{Y}[\epsilon]$. The only non-trivial
component of the latter is $\phi^* \colon \phi^{-1}\mathcal{O}_X
\to \mathcal{O}_\mathfrak{Y}^0 \oplus
\mathcal{O}_\mathfrak{Y}^1\cdot\epsilon$ which we will write as
$\phi^*(f) = \phi^*_0(f) + \phi^*_1(f)\cdot\epsilon$ for
$f\in\mathcal{O}_X$.

Since $\mathcal{O}_X$ is concentrated in degree zero it follows
that $\partial_{\mathfrak{Y}\times\vec{\mathfrak{t}}}\circ\phi^* =
0$ which implies that $\partial_\mathfrak{Y}\phi^*_0(f) =
\phi^*_1(f)$ for $f\in\mathcal{O}_X$. Thus, $\phi^*_1$ is a
derivation, hence give rise to the $\mathcal{O}_X$-linear (in the
sense of $\phi^*_0$) map $\overline{\phi^*_1} \colon
\phi^{-1}\Omega^1_X \to \mathcal{O}_\mathfrak{Y}^1$ making the
diagram
\[
\begin{CD}
\phi^{-1}\Omega^1_X @>{\overline{\phi^*_1}}>> \mathcal{O}_\mathfrak{Y}^1 \\
@A{d}AA @AA{\partial_\mathfrak{Y}}A \\
\phi^{-1}\mathcal{O}_X @>{\phi^*_0}>> \mathcal{O}_\mathfrak{Y}^0
\end{CD}
\]
commutative. Therefore, there exists a unique map of DGA
\[
(\phi^\sharp)^* \colon \phi^{-1}\mathcal{O}_{X^\sharp} = \phi^{-1}\Omega^\bullet_X \to \mathcal{O}_\mathfrak{Y}
\]
whose components in degrees zero and one are equal, respectively,
to $\phi^*_0$ and $\overline{\phi^*_1}$. The pair of maps $\phi
\colon Y \to X$ and $(\phi^\sharp)^*$ define a morphism of DG
manifolds $\phi^\sharp \colon \mathfrak{Y} \to X^\sharp$. The
assignment $\phi \mapsto \phi^\sharp$ defines the map \eqref{Map
to sharp}.

Suppose that $\psi \colon \mathfrak{Y} \to X^\sharp$ is a morphism
of DG manifolds given by a map of manifolds $\psi \colon Y \to X$
and a morphism of DGA $\psi^* \colon
\psi^{-1}\mathcal{O}_{X^\sharp} \to \mathcal{O}_\mathfrak{Y}$. Let
\begin{equation}\label{flat pull-back}
(\psi_\flat)^* \colon \psi^{-1}\mathcal{O}_X \to \mathcal{O}_\mathfrak{Y}[\epsilon]
\end{equation}
denote the map whose only non-trivial component in degree zero is given by $(\psi_\flat)^*(f) = \psi^{*0}(f) + \psi^{*1}(df)\cdot\epsilon$.

\begin{lemma}
The map \eqref{flat pull-back} is a morphism of DGA.
\end{lemma}
\begin{proof}
The map $(\psi_\flat)^*$ is a map of complexes since
\begin{eqnarray*}
\partial_{\mathfrak{Y}\times\vec{\mathfrak{t}}}(\psi^{*0}(f) + \psi^{*1}(df)\cdot\epsilon)
& = & (\partial_{\mathfrak{Y}}\pm\partial_{\epsilon})(\psi^{*0}(f) + \psi^{*1}(df)\cdot\epsilon) \\
& =  & \partial_{\mathfrak{Y}}(\psi^{*0}(f) +
\psi^{*1}(df)\cdot\epsilon) -
\psi^{*1}(df)\cdot\partial_{\epsilon}\epsilon \\
& = & \psi^{*1}(df)-\psi^{*1}(df) \\
& = & 0
\end{eqnarray*}
and a map of algebras since
\begin{eqnarray*}
(\psi_\flat)^*(fg) & =  & \psi^{*0}(fg)+\psi^{*1}(d(fg))\cdot\epsilon
\\
& =  & \psi^{*0}(f)\psi^{*0}(g) + \psi^{*1}(fdg+gdf)\cdot\epsilon \\
& =  & \psi^{*0}(f)\psi^{*0}(g) + \psi^{*1}(fdg)\cdot\epsilon + \psi^{*1}(gdf)\cdot\epsilon \\
& =  & \psi^{*0}(f)\psi^{*0}(g) +
\psi^{*0}(f)\psi^{*1}(dg)\cdot\epsilon +
\psi^{*0}(g)\psi^{*1}(df)\cdot\epsilon \\
& =  & (\psi_{\flat}^{*}f)(\psi_{\flat}^{*}g).
\end{eqnarray*}
\end{proof}

It follows from the lemma above that the pair of maps $\psi \colon
\mathfrak{Y} \to X$ and $(\psi_\flat)^*$ define a morphism of DG
manifolds $\psi_\flat \colon \mathfrak{Y}\times \vec{\mathfrak{t}}
\to X$. The assignment $\psi \mapsto \psi_\flat$ defines the map
\eqref{sharp to Map}.

\begin{proof}[Proof of Theorem \ref{thm: odd path space representable}]
We leave it to the reader to verify that \eqref{Map to sharp} and
\eqref{sharp to Map} are mutually inverse and natural in the
variable $\mathfrak{Y}$. i.e. define an isomorphism of presheaves
$\Map(\vec{\mathfrak{t}}, X) \cong X^\sharp$.
\end{proof}

Let
\begin{equation}\label{eval for odd paths}
\ev^* \colon \mathcal{O}_X \to \mathcal{O}_{X^\sharp}[\epsilon] := \mathcal{O}_{X^\sharp\times \vec{\mathfrak{t}}}
\end{equation}
denote the map defined by $f \mapsto f + df\cdot\epsilon$.

\begin{lemma}
The map \eqref{eval for odd paths} is a morphism of DG-algebras.
\end{lemma}
\begin{proof}
Let $f,g\in\mathcal{O}_X$. The calculation
\begin{eqnarray*}
\ev^{*}(f) \cdot \ev^{*}(g) & = & (f + df\cdot\epsilon) \cdot (g +
dg\cdot\epsilon)\\
& = & fg + fdg\cdot\epsilon + gdf\cdot\epsilon\\
& = & fg + d(fg)\cdot\epsilon = \ev^{*}(fg)
\end{eqnarray*}
shows that \eqref{eval for odd paths} is a morphism of graded algebras.
The computation
\begin{eqnarray*}
\partial_{X^{\sharp}\times \vec{\mathfrak{t}}} (\ev^{*}(f)) & = & \partial_{X^{\sharp}\times \vec{\mathfrak{t}}} (f + df\cdot\epsilon) \\
& = & df - df\cdot\partial_{\epsilon}\epsilon\\
& = & df-df=0.
\end{eqnarray*}
shows that \eqref{eval for odd paths} is a map of complexes.
\end{proof}

Let
\begin{equation}\label{eval odd path manif}
\ev \colon X^\sharp \times \vec{\mathfrak{t}} \to X
\end{equation}
denote the map of DG-manifolds with the structure morphism given by \eqref{eval for odd paths}.

\begin{lemma}
The isomorphism $X^\sharp \to \Map(\vec{\mathfrak{t}}, X)$ given by \eqref{Map to sharp} corresponds to the map \eqref{eval odd path manif}.
\end{lemma}
\begin{proof}
Left to the reader.
\end{proof}

\subsection{Transgression for $\mathcal{O}$-modules}
We denote by $\pr \colon X^\sharp \times \vec{\mathfrak{t}} \to X^\sharp$ the canonical projection. The diagram
\[
\begin{CD}
X^\sharp\times\vec{\mathfrak{t}} @>{\ev}>> X \\
@V{\pr}VV \\
X^\sharp
\end{CD}
\]
gives rise to the functor
\begin{equation}\label{pr-ev for O-mod}
\pr_*\ev^* \colon \Mod{\mathcal{O}_X} \to \Mod{\mathcal{O}_{X^\sharp}}.
\end{equation}
Since the underlying space of both $X^\sharp$ and $X^\sharp\times\vec{\mathfrak{t}}$ is equal to $X$, the functor $\ev^*$ is given by $\ev^*\mathcal{E} = \mathcal{O}_{X^\sharp}[\epsilon]\otimes_{\mathcal{O}_X}\mathcal{E}$ and the effect of the functor $\pr_*$ amounts to restriction of scalars along the unit map $\mathcal{O}_{X^\sharp} \to \mathcal{O}_{X^\sharp}[\epsilon]$.

\begin{lemma}\label{lemma: mult by forms pushout}
For $k \geqslant 0$ the diagram
\[
\begin{CD}
\Omega^k_X \otimes_{\mathcal{O}_X}\Omega^1_X @>>> \Omega^k_X \otimes_{\mathcal{O}_X} \left(\mathcal{O}_{X^\sharp}[\epsilon]\right)^0 \\
@V{\wedge}VV @VVV \\
\Omega^{k+1}_X @>>> \left(\mathcal{O}_{X^\sharp}[\epsilon]\right)^k
\end{CD}
\]
is a push-out square.
\end{lemma}

\begin{remark}\label{remark: prev and jets}
Suppose that $\mathcal{E}$ is a vector bundle on $X$.
The complex $\pr_*\ev^*\mathcal{E}$ with differential denoted $\partial$ may be recognized in more traditional terms as follows.

The differential $\mathcal{E} =
\left(\pr_*\ev^*\mathcal{E}\right)^{-1} \xrightarrow{\partial}
\left(\pr_*\ev^*\mathcal{E}\right)^0$ is a differential operator
of order one and therefore factors as $\mathcal{E}
\xrightarrow{j^1} \mathcal{J}^1(\mathcal{E}) \to
\left(\pr_*\ev^*\mathcal{E}\right)^0$ the second map being
$\mathcal{O}_X$-linear. It is easy to see that it is an
isomorphism.

One deduces easily from Lemma \ref{lemma: mult by forms pushout} and the above isomorphism that for $k \geqslant 0$ the square
\[
\begin{CD}
\Omega^k_X \otimes_{\mathcal{O}_X}\Omega^1_X\otimes_{\mathcal{O}_X}\mathcal{E} @>>> \Omega^k_X \otimes_{\mathcal{O}_X}\mathcal{J}^1(\mathcal{E}) \\
@V{\wedge}VV @VVV \\
\Omega^{k+1}_X\otimes_{\mathcal{O}_X}\mathcal{E} @>>> \left(\pr_*\ev^*\mathcal{E}\right)^k
\end{CD}
\]
is cocartesian. The differential $\partial$ on
$\pr_*\ev^*\mathcal{E}$ is induced by the canonical flat
connection on $\mathcal{J}^\infty(\mathcal{E})$. Namely, the
composition
\[
\Omega^k_X \otimes_{\mathcal{O}_X}\mathcal{J}^\infty(\mathcal{E}) \xrightarrow{\nabla^{can}}\Omega^{k+1}_X\otimes_{\mathcal{O}_X}\mathcal{J}^\infty(\mathcal{E}) \to \Omega^{k+1}_X\otimes_{\mathcal{O}_X}\mathcal{J}^1(\mathcal{E}) \to \left(\pr_*\ev^*\mathcal{E}\right)^{k+1}
\]
factors as
\[
\Omega^k_X \otimes_{\mathcal{O}_X}\mathcal{J}^\infty(\mathcal{E}) \to \Omega^k_X \otimes_{\mathcal{O}_X}\mathcal{J}^1(\mathcal{E}) \to \left(\pr_*\ev^*\mathcal{E}\right)^k \xrightarrow{\partial} \left(\pr_*\ev^*\mathcal{E}\right)^{k+1} .
\]
\end{remark}

\subsection{The ``integration" map}\label{subsection: integration map}
The unique map of complexes
\[
\mathcal{O}_{\vec{\mathfrak{t}}}\otimes_\mathbb{C} \Omega_X^k \to \mathcal{O}_{X^\sharp}[k+1]
\]
whose component of degree $-1$ is the identity map (and, whose
component of degree zero equal, therefore, to the de Rham
differential) extends by linearity to the map
\[
\widetilde{\scriptstyle\int} \colon \mathcal{O}_{X^\sharp}[\epsilon]\otimes_\mathbb{C}\Omega_X^k \to \mathcal{O}_{X^\sharp}[k+1] .
\]
The calculation
\[
\widetilde{\scriptstyle\int}((f+ df\cdot\epsilon)\otimes\omega) =
fd\omega + df\wedge\omega = d(f\omega) =
\widetilde{\scriptstyle\int}(1\otimes f\omega)
\]
shows that the map $\widetilde{\scriptstyle\int}$ factors through the map
\begin{equation}\label{integration map}
{\scriptstyle\int}\colon \pr_*\ev^*\Omega_X^k = \mathcal{O}_{X^\sharp}[\epsilon]\otimes_{\mathcal{O}_X}\Omega_X^k \to \mathcal{O}_{X^\sharp}[k+1] .
\end{equation}

\section{Transgression for Lie algebroids}\label{section: transgression for lie algebroids}

\subsection{Lie algebroids}\label{subsection: Lie algebroids}
An $\mathcal{O}_X$-Lie algebroid structure on an $\mathcal{O}_X$-module $\mathcal{A}$ consists of
\begin{enumerate}
\item a structure of a $\mathbb{C}$-Lie algebra $[\ ,\ ] \colon \mathcal{A}\otimes_\mathbb{C}\mathcal{A} \to \mathcal{A}$;

\item an $\mathcal{O}_X$-linear map $\sigma \colon \mathcal{A} \to
\mathcal{T}_X$ of Lie algebras called \emph{anchor map}.
\end{enumerate}
These data are required to satisfy the compatibility condition (Leibniz rule)
\[
[a,f\cdot b] = \sigma(a)(f)\cdot b + f\cdot [a,b]
\]
for $a,b \in \mathcal{A}$ and $f \in \mathcal{O}_X$.

A morphism of $\mathcal{O}_X$-Lie algebroids $\phi \colon
\mathcal{A}_1 \to \mathcal{A}_2$ is an $\mathcal{O}_X$-linear map
of Lie algebras which commutes with respective anchor maps.

With the above definition of morphisms $\mathcal{O}_X$-Lie algebroids form a category denoted $\LA{\mathcal{O}_X}$.

The notion of Lie algebroid generalizes readily to the DG context.

\begin{remark}\label{remark: bracket in negative degrees}
Suppose that $\mathfrak{X} = (X, \mathcal{O}_\mathfrak{X})$ is a DG-manifold such that $\mathcal{O}_\mathfrak{X}^i = 0$ for $i \leqslant -1$ and $\mathcal{A}$ is a Lie algebroid on $\mathfrak{X}$. Then, the Leibniz rule implies that for $i,j \leqslant -1$ the brackets $[\ ,\ ]^{i,j}\colon \mathcal{A}^i\otimes_\mathbb{C}\mathcal{A}^i \to \mathcal{A}^{i+j}$ are in fact $\mathcal{O}_\mathfrak{X}^0$-bilinear. Indeed, for $f \in \mathcal{O}_\mathfrak{X}^0$ and $a\in \mathcal{A}^i$, $i\leqslant -1$, $\sigma(a)(f)$ has strictly negative degree.

\end{remark}

\subsection{Atiyah algebras}
We recall the definition of the Atiyah algebra of an $\mathcal{O}$-module.

For an $\mathcal{O}_X$-module $\mathcal{E}$ the Atiyah algebra $\mathcal{A}_\mathcal{E}$ is defined as the sheaf whose sections are pairs $(P,\xi)$, where
\begin{itemize}
\item $P \in \shEnd_\mathbb{C}(\mathcal{E})$,
\item $\xi \in \mathcal{T}_X$
\end{itemize}
subject to the condition
\[
\forall f\in \mathcal{O}_X\ \  [P,f] = \xi(f) \in \shEnd_\mathbb{C}(\mathcal{E}) .
\]
In other words, $P$ is a differential operator of order one with
the principal symbol $\sigma(P)$ equal to $\xi\cdot\id$; in
particular, $\xi$ is uniquely determined by $P$. The assignment
$(P,\xi) \mapsto \sigma(P) = \xi$ defines the map $\sigma \colon
\mathcal{A}_\mathcal{E} \to \mathcal{T}_X$. With the obvious
$\mathcal{O}_X$-structure, the bracket given by the commutator of
operators and the principal symbol map as anchor, the Atiyah
algebra acquires the canonical structure of an $\mathcal{O}_X$-Lie
algebroid. If $\mathcal{E}$ is locally free of finite rank, i.e. a
vector bundle, the sequence
\[
0 \to \shEnd_{\mathcal{O}_X}(\mathcal{E}) \to \mathcal{A}_\mathcal{E} \xrightarrow{\sigma} \mathcal{T}_X \to 0
\]
is exact.

The above definition of Atiyah algebra generalizes readily to the DG context (cf. \cite{FN}).

\subsection{Modules over Lie algebroids}
Suppose that $\mathcal{A}$ is an $\mathcal{O}_X$-Lie algebroid.

A structure of an $\mathcal{A}$-module on an
$\mathcal{O}_X$-module $\mathcal{E}$ is a morphism of
$\mathcal{O}_X$-Lie algebroids $\mathcal{A} \to
\mathcal{A}_\mathcal{E}$, i.e. an action of $\mathcal{A}$ on
$\mathcal{E}$ by differential operators of order one which
satisfies the Leibniz rule.

A morphism $\phi \colon \mathcal{E}_1 \to \mathcal{E}_2$ of $\mathcal{A}$-modules is an $\mathcal{O}_X$-linear map which commutes with respective actions.

With the above definition of morphisms $\mathcal{A}$-modules form a category denoted $\Mod{\mathcal{A}}$.

\subsection{Vector fields on the odd path space}
The canonical map of Lie algebras $L \colon
\mathcal{T}_X \to \mathcal{T}_{X^\sharp}$ defined by $\xi \mapsto
L_\xi$ which  extends to the map of graded Lie algebras $L \colon
\mathcal{O}_{\vec{\mathfrak{t}}}\otimes_\mathbb{C} \mathcal{T}_X
\to \mathcal{T}_{X^\sharp}$ by setting $L_{\epsilon\cdot\xi} =
\iota_\xi$ and, by linearity, to the map of
$\mathcal{O}_{X^\sharp}$-Lie algebroids
\begin{equation}\label{cone to T}
\widetilde{L} \colon \mathcal{O}_{X^\sharp}[\epsilon]\otimes_\mathbb{C}\mathcal{T}_X \to \mathcal{T}_{X^\sharp}
\end{equation}

\begin{lemma}
{~}
\begin{enumerate}
\item The map \eqref{cone to T} factors as
\[
\mathcal{O}_{X^\sharp}[\epsilon]\otimes_\mathbb{C}\mathcal{T}_X \to  \mathcal{O}_{X^\sharp}[\epsilon]\otimes_{\mathcal{O}_X}\mathcal{T}_X = \pr_*\ev^*\mathcal{T}_X  \to \mathcal{T}_{X^\sharp}.
\]

\item The map $\pr_*\ev^*\mathcal{T}_X  \to \mathcal{T}_{X^\sharp}$ is an isomorphism of $\mathcal{O}_{X^\sharp}$-Lie algebroids.
\end{enumerate}
\end{lemma}
\begin{proof}
The calculation
\[
\widetilde{L}((f+ df\cdot\epsilon)\otimes \xi) = f\cdot L_\xi +
df\cdot\iota_\xi = L_{f\xi}
\]
implies the first claim.

Let $\mathcal{T}_{X^\sharp/X}\subset\mathcal{T}_{X^\sharp}$ denote
the normalizer of $\mathcal{O}_X\subset\mathcal{O}_{X^\sharp}$.
The map of graded (but not DG) manifolds $p \colon X^\sharp \to X$
which corresponds to the inclusion $\mathcal{O}_X \hookrightarrow
\mathcal{O}_{X^\sharp}$ gives rise to the short exact sequence of
$\mathcal{O}_{X^\sharp}$-modules
\[
0 \to \mathcal{T}_{X^\sharp/X} \to \mathcal{T}_{X^\sharp} \to \mathcal{O}_{X^\sharp}\otimes_{\mathcal{O}_X} \mathcal{T}_X \to 0 .
\]
Since the action of $\epsilon\otimes\mathcal{T}_X$ on
$\mathcal{O}_{X^\sharp}$ is $\mathcal{O}_X$-linear  the
restriction of the map $\pr_*\ev^*\mathcal{T}_X  \to
\mathcal{T}_{X^\sharp}$ to $\epsilon\cdot\pr_*\ev^*\mathcal{T}_X$
takes values in $\mathcal{T}_{X^\sharp/X}$ and, hence, induces a
morphism of short exact sequences
\[
\begin{CD}
0 @>>> \epsilon\cdot\pr_*\ev^*\mathcal{T}_X @>>> \pr_*\ev^*\mathcal{T}_X @>>> \pr_*\ev^*\mathcal{T}_X/\epsilon\cdot\pr_*\ev^*\mathcal{T}_X  @>>> 0 \\
& & @VVV @VVV @VVV \\
0 @>>> \mathcal{T}_{X^\sharp/X} @>>> \mathcal{T}_{X^\sharp} @>>> \mathcal{O}_{X^\sharp}\otimes_{\mathcal{O}_X} \mathcal{T}_X @>>> 0
\end{CD}
\]
It is easy to see that the left vertical map (the
$\mathcal{O}_{X^\sharp}$-linear extension of the action by
interior multiplication)
$\mathcal{O}_{X^\sharp}\otimes_{\mathcal{O}_X} \mathcal{T}_X[1]
\to \mathcal{T}_{X^\sharp/X}$ is an isomorphism, and so is the
right vertical map. This implies the second claim.
\end{proof}

\subsection{Transgression for Lie algebroids}\label{subsection: Transgression for Lie algebroids}
Suppose that $\mathcal{A}$ is an $\mathcal{O}_X$-Lie algebroid as in \ref{subsection: Lie algebroids}.

The sheaf $\mathcal{O}_{X^\sharp}[\epsilon]\otimes_\mathbb{C}\mathcal{A}$ has a canonical structure of an $\mathcal{O}_{X^\sharp}$-Lie algebroid with
\begin{itemize}
\item the anchor map is given by the composition
\[
\widetilde{\sigma}:\mathcal{O}_{X^\sharp}[\epsilon]\otimes_\mathbb{C}\mathcal{A}
\xrightarrow{\id\otimes\sigma}
\mathcal{O}_{X^\sharp}[\epsilon]\otimes_\mathbb{C}\mathcal{T}_X
\xrightarrow{\eqref{cone to T}} \mathcal{T}_{X^\sharp} ,
\]

\item the bracket defined the extension of the bracket on
$\mathcal{A}$ by the Leibniz rule, the anchor above and the action
of $\mathcal{T}_{X^\sharp}$ on $\mathcal{O}_{X^\sharp}[\epsilon]$;
namely:
\begin{align*}
[\omega_1\otimes1\otimes a_1,\, & \omega_2\otimes1\otimes a_2] =
\omega_1\wedge L_{\sigma(a_1)}\omega_2\otimes1\otimes a_2 \\
&
+(-1)^{ij}\omega_2\wedge(\omega_1\otimes1\otimes[a_1,a_2]-L_{\sigma(a_2)}\omega_1\otimes1\otimes
a_1),
\end{align*}

\begin{align*}
[\omega_1\otimes1\otimes a_1,\, & \omega_2\otimes\epsilon\otimes
b_2] = \omega_1\wedge
L_{\sigma(a_1)}\omega_2\otimes\epsilon\otimes b_2 \\
&
+(-1)^{ij}\omega_2\wedge(\omega_1\otimes\epsilon\otimes[a_1,b_2]-(-1)^i\iota_{\sigma(b_2)}\omega_1\otimes1\otimes
a_1),
\end{align*}

\begin{align*}
[\omega_1\otimes\epsilon\otimes b_1,\, &
\omega_2\otimes\epsilon\otimes b_2] = \\
& \omega_1\wedge\iota_{\sigma(b_1)}\omega_2\otimes\epsilon\otimes
b_2 +
(-1)^{ij+j+i}\omega_2\wedge\iota_{\sigma(b_2)}\omega_1\otimes\epsilon\otimes
b_1.
\end{align*}
where
$\omega_1\in\mathcal{O}^{i}_{X^{\sharp}},\,\omega_2\in\mathcal{O}^{j}_{X^{\sharp}}$,
$a_1,a_2,b_1,b_2\in\mathcal{A}$.
\end{itemize}

\begin{lemma}\label{lemma: LA structure on pr-ev}
The structure of an $\mathcal{O}_{X^\sharp}$-Lie algebroid on
$\mathcal{O}_{X^\sharp}[\epsilon]\otimes_\mathbb{C}\mathcal{A}$
descends to one on $\pr_*\ev^*\mathcal{A}$.
\end{lemma}
\begin{proof}
The proof is given in \ref{subsection: proof of LA structure descends}.
\end{proof}

\begin{notation}
In what follows we shall denote the $\mathcal{O}_{X^\sharp}$-Lie
algebroid structure on $\pr_*\ev^*\mathcal{A}$ given by Lemma
\ref{lemma: LA structure on pr-ev} by $\mathcal{A}^\sharp$.
\end{notation}

We leave it to the reader to check that the morphism
$\pr_*\ev^*\phi \colon \pr_*\ev^*\mathcal{A}_1 \to
\pr_*\ev^*\mathcal{A}_2$ induced by a morphism $\phi \colon
\mathcal{A}_1 \to \mathcal{A}_2$ of $\mathcal{O}_X$-Lie algebroids
is in fact a morphism of $\mathcal{O}_{X^\sharp}$-Lie algebroids.
That is, the assignment $\mathcal{A} \mapsto \mathcal{A}^\sharp$
extends to a functor
\[
\left(\ \right)^\sharp \colon \LA{\mathcal{O}_X} \to \LA{\mathcal{O}_{X^\sharp}} .
\]

\subsection{Transgression for modules}
Suppose that $\mathcal{E}$ is an $\mathcal{O}_X$-module. For $D\in
\mathcal{A}_\mathcal{E}$ we denote by $\iota_D \colon
\left(\pr_*\ev^*\mathcal{E}\right)^0 \to \mathcal{E}$ the unique
$\mathcal{O}_X$-linear map such that $D = \iota_D \circ \partial$,
see Remark \ref{remark: prev and jets}. Let
\begin{equation}\label{iota D}
\widetilde{\iota_D} \colon \Omega^{k+1}_X \otimes_{\mathcal{O}_X} \left(\pr_*\ev^*\mathcal{E}\right)^0 \to \left(\pr_*\ev^*\mathcal{E}\right)^k
\end{equation}
denote the map defined by the formula
\[
\widetilde{\iota_D}(\alpha\otimes B) = \iota_{\sigma(D)}\alpha\otimes B + (-1)^{\alpha}\alpha\otimes\iota_D B .
\]

\begin{lemma}\label{lemma: iota descends}
{~}
\begin{enumerate}
\item The map \eqref{iota D} factors through the map
$\widetilde{\iota_D} \colon
\left(\pr_*\ev^*\mathcal{E}\right)^{k+1} \to
\left(\pr_*\ev^*\mathcal{E}\right)^k$.

\item The $\mathcal{O}_X$-linear map $\widetilde{\iota_D} \colon
\pr_*\ev^*\mathcal{E} \to \pr_*\ev^*\mathcal{E}[-1]$ is a
differential operator or order one with respect to the
$\mathcal{O}_{X^\sharp}$-module structure with principal symbol
$\sigma(D)\otimes\epsilon$. In particular, $\widetilde{\iota_D}
\in \mathcal{A}^{-1}_{\pr_*\ev^*\mathcal{E}}$.
\end{enumerate}
\end{lemma}
\begin{proof}
The proof is given in \ref{subsection: proof of iota descends}.
\end{proof}

Let
\[
\widetilde{D} = [\partial, \widetilde{\iota_D}] .
\]
Thus, $\widetilde{D} \colon \pr_*\ev^*\mathcal{E} \to
\pr_*\ev^*\mathcal{E}$ is a differential operator or order one
with respect to the $\mathcal{O}_{X^\sharp}$-module structure with
principal symbol $L_{\sigma(D)} = [\partial, \sigma(D)\epsilon]$.
In particular, $\widetilde{D} \in
\mathcal{A}^0_{\pr_*\ev^*\mathcal{E}}$. Thus, the assignment $D
\mapsto \widetilde{D}$ defines a map
\begin{equation}\label{LA action on transgression}
\mathcal{A}_\mathcal{E} \to \mathcal{A}_{\pr_*\ev^*\mathcal{E}} .
\end{equation}

\begin{lemma}\label{lemma: Atiyah action}
{~}
\begin{enumerate}
\item For $D\in\mathcal{A}_\mathcal{E}$, $A\in\mathcal{O}_{X^\sharp}[\epsilon]$ and $e\in\mathcal{E}$
\[
\widetilde{D}(A\otimes e) = L_{\sigma(D)} A\otimes e + A\otimes D(e)
\]

\item The map $\mathcal{A}_\mathcal{E} \xrightarrow{\eqref{LA action on transgression}} \mathcal{A}_{\pr_*\ev^*\mathcal{E}}$ is a morphism
of Lie algebras.

\item The diagram
\begin{equation}\label{extension of Atiyah action commutes with the anchors}
\begin{CD}
\mathcal{A}_\mathcal{E} @>{\eqref{LA action on transgression}}>> \mathcal{A}_{\pr_*\ev^*\mathcal{E}} \\
@V{\sigma}VV @VV{\sigma}V \\
\mathcal{T}_X @>{L}>> \mathcal{T}_{X^\sharp}
\end{CD}
\end{equation}
is commutative.

\item The induced map $\shEnd_{\mathcal{O}_X}(\mathcal{E}) \to
\shEnd_{\mathcal{O}_{X^\sharp}}(\pr_*\ev^*\mathcal{E})$ between
the kernels of respective anchor maps coincides with the map
induced by the functor $\pr_*\ev^*$.
\end{enumerate}
\end{lemma}
\begin{proof}
The proof is given in \ref{subsection: proof of Atiyah action}
\end{proof}

Let
\begin{equation}\label{pre Lie-E}
L^\mathcal{E} \colon \mathcal{O}_{X^\sharp}[\epsilon]\otimes_\mathbb{C}\mathcal{A}_\mathcal{E} \to \mathcal{A}_{\pr_*\ev^*\mathcal{E}}
\end{equation}
denote the $\mathcal{O}_{X^\sharp}$-linear extension of the map
$L^\mathcal{E} \colon
\mathbb{C}[\epsilon]\otimes_\mathbb{C}\mathcal{A}_\mathcal{E} \to
\mathcal{A}_{\pr_*\ev^*\mathcal{E}}$ defined by $D \mapsto
L^\mathcal{E}_D = \widetilde{D}$ and $\epsilon\cdot D \mapsto
L^\mathcal{E}_{\epsilon\cdot D} = \widetilde{\iota_D}$.

\begin{lemma}\label{lemma: morphism of LA}
{~}
\begin{enumerate}
\item The map \eqref{pre Lie-E} descends to the map
\begin{equation}\label{Lie-E}
L^\mathcal{E} \colon \mathcal{A}_\mathcal{E}^\sharp \to \mathcal{A}_{\pr_*\ev^*\mathcal{E}} .
\end{equation}

\item The map \eqref{Lie-E} is a morphism of $\mathcal{O}_{X^\sharp}$-Lie algebroids.
\end{enumerate}
\end{lemma}
\begin{proof}
The proof is given in \ref{subsection: proof of morphism of LA}.
\end{proof}

Suppose that $\mathcal{A}$ is an $\mathcal{O}_X$-Lie algebroid and
$\mathcal{E}$ is an $\mathcal{A}$-module with the module structure
given by the morphism $\alpha \colon \mathcal{A} \to
\mathcal{A}_\mathcal{E}$. The composition $\mathcal{A}^\sharp
\xrightarrow{\alpha^\sharp} \mathcal{A}_\mathcal{E}^\sharp
\xrightarrow{\eqref{Lie-E}} \mathcal{A}_{\pr_*\ev^*\mathcal{E}}$
endows $\pr_*\ev^*\mathcal{E}$ with a canonical structure of an
$\mathcal{A}^\sharp$-module. Thus, the functor \eqref{pr-ev for
O-mod} induces the functor
\[
\pr_*\ev^* \colon \Mod{\mathcal{A}} \to \Mod{\mathcal{A}^\sharp} .
\]

\subsection{Marked Lie algebroids}
Suppose that $\mathfrak{X} = (X, \mathcal{O}_\mathfrak{X})$ is a DG-manifold.

\begin{lemma}\label{lemma: marking in ker of anchor}
Suppose that $\mathcal{A}$ is a $\mathcal{O}_\mathfrak{X}$-Lie
algebroid. Any (homogeneous) central section
$\mathfrak{c}\in\Gamma(X;\mathcal{A})$ (i.e. $[\mathfrak{c},
\mathcal{A}] = 0$) belongs to the kernel of the anchor map.
\end{lemma}
\begin{proof}
For $a\in\mathcal{A}$, $\gamma\in\mathcal{O}_{X^{\sharp}}$, the calculation
\[
0=[\mathfrak{c},\gamma\cdot a]=\sigma(\mathfrak{c})(\gamma)\cdot a
\]
implies that $\sigma(\mathfrak{c})\gamma=0$. Applying the anchor
map we find that $\sigma(\mathfrak{c})(\gamma)\cdot \sigma(a) = 0$
for all $a\in\mathcal{A}$ and $\gamma\in\mathcal{O}_{X^{\sharp}}$.
With $a = \mathfrak{c}$ this means that the vector field
$\sigma(\mathfrak{c})$ annihilates $\mathcal{O}_{X^{\sharp}}$.
\end{proof}

\begin{definition}\label{defn: marked Lie algd}
A \emph{marked $\mathcal{O}_\mathfrak{X}$-Lie algebroid} is a pair
$(\mathcal{A}, \mathfrak{c})$, where $\mathcal{A}$ is a
$\mathcal{O}_\mathfrak{X}$-Lie algebroid and
$\mathfrak{c}\in\Gamma(X;\mathcal{A})$ is a homogeneous central
section.
\end{definition}

A morphism $\phi\colon (\mathcal{A}_1, \mathfrak{c}_1) \to
(\mathcal{A}_2, \mathfrak{c}_2)$ is a morphism of Lie algebroids
$\phi\colon \mathcal{A}_1 \to \mathcal{A}_2$ such that
$\phi(\mathfrak{c}_1) = \mathfrak{c}_2$. In particular,
$\mathfrak{c}_1$ and $\mathfrak{c}_2$ have the same degree.

With the above definitions marked $\mathcal{O}_\mathfrak{X}$-Lie
algebroids and morphisms thereof form a category denoted
$\LA{\mathcal{O}_\mathfrak{X}}^\star$. The full subcategory of
marked $\mathcal{O}_\mathfrak{X}$-Lie algebroids $(\mathcal{A},
\mathfrak{c})$ with $\deg\mathfrak{c} = k$ is denoted
$\LA{\mathcal{O}_\mathfrak{X}}^\star_k$.

For a marked Lie algebroid $(\mathcal{A}, \mathfrak{c})$ with $\deg\mathfrak{c} = n$ let
\[
\overline{\mathcal{A}} := \coker(\mathcal{O}_\mathfrak{X}[n] \xrightarrow{\cdot\mathfrak{c}} \mathcal{A}) .
\]

\begin{lemma}
The structure of a Lie algebroid on $\mathcal{A}$ descends to one on $\overline{\mathcal{A}}$.
\end{lemma}
\begin{proof}
Since $\mathfrak{c}$ is central the submodule
$\mathcal{O}_{\mathfrak{X}}\cdot\mathfrak{c}$
satisfies
$[\mathcal{O}_{\mathfrak{X}}\cdot\mathfrak{c},
\mathcal{A}]\subseteq
\mathcal{O}_{\mathfrak{X}}\cdot\mathfrak{c}$. Moreover, by Lemma
\ref{lemma: marking in ker of anchor} the anchor map factors
through $\overline{\mathcal{A}}$.
\end{proof}

\subsection{$\mathcal{O}_\mathfrak{X}[n]$-extensions}
Suppose that $\mathcal{B}$ is a $\mathcal{O}_\mathfrak{X}$-Lie algebroid.

\begin{definition}
An $\mathcal{O}_\mathfrak{X}[n]$-extension of $\mathcal{B}$ is a
marked $\mathcal{O}_\mathfrak{X}$-Lie algebroid $(\mathcal{A},
\mathfrak{c})$ with $\deg\mathfrak{c} = n$ together with the
identification $\overline{\mathcal{A}} \cong \mathcal{B}$ such
that the sequence
\[
0 \to \mathcal{O}_\mathfrak{X}[n] \xrightarrow{\cdot\mathfrak{c}} \mathcal{A} \to \mathcal{B} \to 0
\]
is exact.
\end{definition}
A morphism of $\mathcal{O}_\mathfrak{X}[n]$-extensions of
$\mathcal{B}$ is a morphism of marked Lie algebroids which induces
the identity map on $\mathcal{B}$.

We denote the category of $\mathcal{O}_\mathfrak{X}[n]$-extensions of $\mathcal{B}$ by $\OExt{n}(\mathcal{B})$.

It follows from the definitions that a morphism of
$\mathcal{O}_\mathfrak{X}[n]$-extensions of $\mathcal{B}$ induces
a morphism of extensions of $\mathcal{B}$ by
$\mathcal{O}_\mathfrak{X}[n]$ (in the abelian category of
$\mathcal{O}_\mathfrak{X}$-modules), hence is an isomorphism of
$\mathcal{O}_\mathfrak{X}$-modules. The inverse map is easily seen
to be a morphism of $\mathcal{O}_\mathfrak{X}[n]$-extensions. It
follows that the category $\OExt{n}(\mathcal{B})$ is a groupoid.

\section{Courant algebroids}\label{section: courant algebroids}

\begin{definition}\label{defn: courant}
A $k$-dimensional Courant algebroid is an $\mathcal{O}_X$-module $\mathcal{Q}$ equipped with
\begin{enumerate}
\item a structure of a Leibniz $\mathbb{C}$-algebra
\[
\{\ ,\ \} \colon \mathcal{Q}\otimes_\mathbb{C}\mathcal{Q} \to \mathcal{Q} ;
\]

\item an $\mathcal{O}_X$-linear map of Leibniz algebras (the anchor map)
\[
\pi \colon \mathcal{Q} \to \mathcal{T}_X ;
\]

\item a symmetric $\mathcal{O}_X$-bilinear pairing
\[
\ip\colon \mathcal{Q}\otimes_{\mathcal{O}_X}\mathcal{Q} \to \Omega_X^{k-1} ;
\]

\item an $\mathcal{O}_X$-linear map
\[
\pi^\dagger \colon \Omega_X^k \to \mathcal{Q} .
\]
\end{enumerate}
These data are required to satisfy
\begin{eqnarray}
\pi\circ\pi^\dagger & = & 0 \label{complex}\\
\{q_1, fq_2\} & = & f\{q_1,q_2\} + \pi(q_1)(f)q_2 \label{leibniz}\\
\langle \{q,q_1\}, q_2 \rangle + \langle q_1, \{q, q_2\}\rangle & = & L_{\pi(q)}\langle q_1,q_2 \rangle \label{ip invariance}\\
\iota_{\pi(q)} \langle q_1,q_2 \rangle & = &  -\iota_{\pi(q_1)} \langle q,q_2 \rangle - \iota_{\pi(q_2)} \langle q,q_1 \rangle \label{ip jacobi}\\
\{q, \pi^\dagger(\alpha)\} & = & \pi^\dagger(L_{\pi(q)}(\alpha)) \label{forms left ideal}\\
\langle q, \pi^\dagger(\alpha) \rangle & = & \iota_{\pi(q)}\alpha \label{adjunction}\\
\{q_1, q_2\} + \{q_2, q_1\} & = & \pi^\dagger(d\langle q_1, q_2\rangle) \label{symmetrizer}
\end{eqnarray}
for $f\in \mathcal{O}_X$ and $q, q_1, q_2 \in \mathcal{Q}$.
\end{definition}

A morphism $\phi \colon \mathcal{Q}_1 \to \mathcal{Q}_2$ of
$k$-dimensional Courant algebroids on $X$ is an
$\mathcal{O}_X$-linear map of Leibniz $\mathbb{C}$-algebras such
that the diagram
\[
\begin{CD}
\Omega^k_X @>{\pi^\dagger_1}>> \mathcal{Q}_1 @>{\pi_1}>> \mathcal{T}_X \\
@| @V{\phi}VV @| \\
\Omega^k_X @>{\pi^\dagger_2}>> \mathcal{Q}_2 @>{\pi_2}>> \mathcal{T}_X
\end{CD}
\]
is commutative.

With the above definitions $k$-dimensional Courant algebroids on
$X$ and morphisms thereof form a category henceforth denoted
$\CA(X)_k$.

\begin{remark}
A zero-dimensional Courant algebroid is just a marked $\mathcal{O}_X$-Lie algebroid with the marking given by $\pi^\dagger$.

A one-dimensional Courant algebroid is a Courant algebroid in the sense of \cite{B} (and, in restricted setting, of \cite{LWX}, \cite{R} and \cite{S3}). In this case the axiom \eqref{ip jacobi} is redundant.
\end{remark}

\begin{remark}
The formula \eqref{ip jacobi} says that the graded
$\mathcal{O}_X$-module $\mathcal{Q}[1] \oplus
\bigoplus\limits_{i=1}^{k}\Omega_X^{k-i}[i+1]$ equipped with the
bracket whose only non-zero component is $\ip \colon
\mathcal{Q}[1]\otimes\mathcal{Q}[1] \to \Omega_X^{k-1}[2]$ is a
graded Lie algebra.
\end{remark}

\subsection{The associated Lie algebroid}
Suppose that $\mathcal{Q}$ is a $k$-dimensional Courant algebroid on $X$. Let
\[
\overline{\mathcal{Q}} = \coker(\pi^\dagger) .
\]
The formula \eqref{forms left ideal} implies that
$\{\mathcal{Q},\im(\pi^\dagger)\} \subseteq \im(\pi^\dagger)$.
Formulas \eqref{forms left ideal}, \eqref{symmetrizer} and
\eqref{adjunction} imply that $\{\im(\pi^\dagger),\mathcal{Q}\}
\subseteq \im(\pi^\dagger)$. Therefore, the bracket on
$\mathcal{Q}$ descends to a bilinear operation
\begin{equation}\label{assoc lie bracket}
\{\ ,\ \} \colon \overline{\mathcal{Q}}\otimes_\mathbb{C}\overline{\mathcal{Q}} \to \overline{\mathcal{Q}} .
\end{equation}
The condition \eqref{complex} implies that the anchor map $\pi \colon \mathcal{Q} \to \mathcal{T}_X$ factors through
\begin{equation}\label{assoc lie anchor}
\pi \colon \overline{\mathcal{Q}} \to \mathcal{T}_X
\end{equation}

\begin{lemma}
The bracket \eqref{assoc lie bracket} and the anchor \eqref{assoc lie anchor} determine a structure of a Lie algebroid on $\overline{\mathcal{Q}}$.
\end{lemma}
\begin{proof}
According to \eqref{symmetrizer} the symmetrization of the Leibniz
bracket on $\mathcal{Q}$ takes values in the image of
$\pi^\dagger$. Therefore the induced bracket \eqref{assoc lie
bracket} is skew-symmetric. The Leibniz rule and the Jacobi
identity for $\overline{\mathcal{Q}}$ follow from those for
$\mathcal{Q}$.
\end{proof}

In what follows we refer to the Lie algebroid $\overline{\mathcal{Q}}$ as \emph{the Lie algebroid associated to the Courant algebroid $\mathcal{Q}$}.

The assignment $\mathcal{Q} \mapsto \overline{\mathcal{Q}}$ extends to a functor
\[
\overline{(\ )} \colon \CA(X)_k \longrightarrow \LA{\mathcal{O}_X} .
\]

\subsection{Courant extensions}
\begin{definition}\label{def: courant extension}
For $\mathcal{A} \in \LA{\mathcal{O}_X}$ a \emph{($k$-dimensional)
Courant extension of $\mathcal{A}$} is a $k$-dimensional Courant
algebroid $\mathcal{Q}$ together with the identification
$\overline{\mathcal{Q}} \cong \mathcal{A}$ such that the sequence
\begin{equation}\label{CExt transitive exact sequence}
0 \to \Omega^k_X \xrightarrow{\pi^\dagger} \mathcal{Q} \to \mathcal{A} \to 0
\end{equation}
is exact.
\end{definition}

A morphism $\phi\colon \mathcal{Q}_1 \to \mathcal{Q}_2$ of Courant
extensions of $\mathcal{A}$ is a morphism of Courant algebroids
which is compatible with the identifications
$\overline{\mathcal{Q}_i} \cong \mathcal{A}$.

\begin{notation}
We denote the category of $k$-dimensional Courant extension of $\mathcal{A}$ by $\CExt{\mathcal{A}}{k}$.
\end{notation}

A morphism in $\CExt{\mathcal{A}}{k}$ induces a morphism of
associated short exact sequences \eqref{CExt transitive exact
sequence}, hence is an isomorphism of  underlying
$\mathcal{O}_X$-modules. It is easy to see that the inverse map
is, in fact, a morphism of Courant algebroids. Consequently,
$\CExt{\mathcal{A}}{k}$ is a groupoid.

\subsection{Transitive Courant algebroids}
A Courant algebroid $\mathcal{Q}$ is \emph{transitive} if the
associated Lie algebroid $\overline{\mathcal{Q}}$ is transitive,
i.e. if the anchor map is surjective.

\begin{lemma}\label{lemma: transitive injective coanchor}
Suppose that $\mathcal{Q}$ is a transitive Courant algebroid.
Then, the map $\Omega^k_X \xrightarrow{\pi^\dagger} \mathcal{Q}$
is injective, hence the sequence
\[
0 \to \Omega^k_X \xrightarrow{\pi^\dagger} \mathcal{Q} \to \overline{\mathcal{Q}} \to 0
\]
is exact.
\end{lemma}
\begin{proof}
According to \eqref{adjunction} the map $\pi^\dagger$ is adjoint
to the map $\pi$. As the latter is surjective, the former is
injective. Namely, if $\alpha\in\Omega^k_X$ satisfies
$\pi^\dagger(\alpha) = 0$, then $\iota_{\pi(q)}\alpha = 0$ for all
$q\in\mathcal{Q}$. Since $\pi$ is surjective $\iota_{\xi}\alpha =
0$ for all $\xi\in\mathcal{T}_X$, i.e. $\alpha = 0$.
\end{proof}

\subsection{Exact Courant algebroids}
\begin{definition}
A Courant algebroid $\mathcal{Q}$ is \emph{exact} if the anchor map $\overline{\mathcal{Q}}\xrightarrow{\pi} \mathcal{T}_X$ is an isomorphism.
\end{definition}

Equivalently, by Lemma \ref{lemma: transitive injective coanchor}
an exact Courant algebroid $\mathcal{Q}$ is a Courant extension of
$\mathcal{T}_X$, i.e. the sequence
\begin{equation}\label{ECA exact sequence}
0 \to \Omega^k_X \xrightarrow{\pi^\dagger} \mathcal{Q}\xrightarrow{\pi} \mathcal{T}_X \to 0
\end{equation}
is exact.

We denote by $\ECA(X)_k$ the category (groupoid) of exact Courant algebroids.

\begin{example}\label{example: trivial ECA}
Let $\mathcal{Q}_0 = \Omega^k_X\oplus\mathcal{T}_X$. Let $\pi
\colon \mathcal{Q}_0 \to \mathcal{T}_X$ (respectively,
$\pi^\dagger \colon \Omega^k_X \to \mathcal{Q}_0$) be the
projection on the second factor (respectively, the inclusion of
the first summand). Let $\langle\ ,\ \rangle\colon
\mathcal{Q}_0\otimes_{\mathcal{O}_X}\mathcal{Q}_0 \to
\Omega^{k-1}$ be the symmetric pairing with
$\langle\Omega^k_X,\Omega^k_X\rangle = 0$,
$\langle\mathcal{T}_X,\mathcal{T}_X\rangle = 0$ and $\langle\ ,\
\rangle\colon \mathcal{T}_X\otimes_{\mathcal{O}_X}\Omega^k_X \to
\Omega^{k-1}$ the interior multiplication. There is a unique
structure of a Courant algebroid on $\mathcal{Q}_0$ characterized
by the property that the canonical inclusion of $\mathcal{T}_X$
into $\mathcal{Q}_0$ is a map of Leibniz algebras.
\end{example}

Suppose that $\mathcal{Q}$ is an exact Courant algebroid. A
splitting $\nabla \colon \mathcal{T}_X \to \mathcal{Q}$ of
\eqref{ECA exact sequence} (i.e. a section of the anchor map) is
called a \emph{connection (on $\mathcal{Q}$)} if is isotropic with
respect to the symmetric pairing, i.e. $\langle
\nabla(\xi),\nabla(\eta)\rangle = 0$ for all $\xi, \eta \in
\mathcal{T}_X$. We denote by $\Conn(\mathcal{Q})$ the sheaf of
locally defined connections on $Q$.

\begin{lemma}\label{lemma: conn locally non-empty}
The sheaf $\Conn(\mathcal{Q})$ is locally non-empty.
\end{lemma}
\begin{proof}
Let $s \colon \mathcal{T}_X \to \mathcal{Q}$ denote a locally
defined splitting of \eqref{ECA exact sequence}. Let $\phi \colon
\mathcal{T}_X \to \Omega^k_X$ denote the map uniquely determined
by
\[
\iota_\eta\phi(\xi) = -\frac12 \langle s(\xi),s(\eta)\rangle \ .
\]
Let $\nabla = s + \pi^\dagger\circ\phi$. Then,
\begin{multline*}
\langle \nabla(\xi),\nabla(\eta)\rangle = \langle s(\xi) + \pi^\dagger(\phi(\xi)), s(\eta) + \pi^\dagger(\phi(\eta)) \rangle = \\
\langle s(\xi), s(\eta) \rangle + \iota_\eta\phi(\xi) + \iota_\xi\phi(\eta) = \\
\langle s(\xi), s(\eta) \rangle -\frac12 \langle s(\xi),s(\eta)\rangle -\frac12 \langle s(\eta),s(\xi)\rangle = 0 \ .
\end{multline*}
\end{proof}

For a connection $\nabla\in\Conn(\mathcal{Q})$ and
$\omega\in\Omega^{k+1}_X$ let $\nabla + \omega \colon
\mathcal{T}_X \to \mathcal{Q}$ denote the map defined by
\[
(\nabla + \omega)(\xi) = \nabla(\xi) + \pi^\dagger(\iota_\xi\omega) \ .
\]
\begin{lemma}
{~}
\begin{enumerate}
\item For $\nabla\in\Conn(\mathcal{Q})$ and $\omega\in\Omega^{k+1}_X$ the map $\nabla + \omega$ is a connection on $\mathcal{Q}$.

\item The assignment $\nabla \mapsto \nabla + \omega$ is an action
of (the sheaf of groups) $\Omega^{k+1}_X$ on $\Conn(\mathcal{Q})$
which endows the latter with a structure of a
$\Omega^{k+1}_X$-torsor.
\end{enumerate}
\end{lemma}
\begin{proof}
For $\xi, \eta \in \mathcal{T}_X$
\begin{multline*}
\langle (\nabla + \omega)(\xi),(\nabla + \omega)(\eta)\rangle = \langle \nabla(\xi) + \pi^\dagger(\iota_\xi\omega), \nabla(\eta) + \pi^\dagger(\iota_\eta\omega) \rangle = \\ \langle \nabla(\xi), \nabla(\eta)\rangle + \langle \pi^\dagger(\iota_\xi\omega), \nabla(\eta) \rangle + \langle \nabla(\xi), \pi^\dagger(\iota_\eta\omega) \rangle + \langle \pi^\dagger(\iota_\xi\omega)), \pi^\dagger(\iota_\eta\omega) \rangle
\end{multline*}
with $\langle \nabla(\xi), \nabla(\eta)\rangle = 0$ because
$\nabla$ is a connection, $\langle \pi^\dagger(\iota_\xi\omega)),
\pi^\dagger(\iota_\eta\omega) \rangle = 0$ by \eqref{adjunction}
and $\langle \nabla(\xi), \pi^\dagger(\iota_\eta\omega) \rangle =
\iota_\xi\iota_\eta\omega = -\iota_\eta\iota_\xi\omega = -\langle
\pi^\dagger(\iota_\xi\omega), \nabla(\eta) \rangle$ by
skew-symmetry of $\omega$ and \eqref{adjunction}. Hence, $\langle
(\nabla + \omega)(\xi),(\nabla + \omega)(\eta)\rangle = 0$ which
proves the first claim.

It is clear that the assignment $(\omega, \nabla) \mapsto \nabla +
\omega$ defines an action of $\Omega^{k+1}_X$ on
$\Conn(\mathcal{Q})$. For $\nabla_1, \nabla_2 \in
\Conn(\mathcal{Q})$ the difference $\nabla_1 - \nabla_2$ satisfies
$\pi\circ(\nabla_1 - \nabla_2) = 0$, hence defines a map $\phi
\colon \mathcal{T}_X \to \Omega^k_X$ (by $\pi^\dagger\circ\phi =
\nabla_1 - \nabla_2$) or, equivalently, a section of
$\Omega^1_X\otimes_{\mathcal{O}_X} \Omega^k_X$. The calculation
above shows that the latter is totally skew-symmetric, i.e. a
section of $\Omega^{k+1}_X$. By Lemma \ref{lemma: conn locally
non-empty}, $\Conn(\mathcal{Q})$ is locally non-empty, hence a
torsor.
\end{proof}

For a connection $\nabla$ on $\mathcal{Q}$ and $\xi, \eta \in
\mathcal{T}_X$ the expression $\{\nabla(\xi),\nabla(\eta)\} -
\nabla([\xi,\eta])$ belong to the kernel of the anchor map. Hence
there is a unique $c(\nabla)(\xi, \eta)\in \Omega^k_X$ with
\[
\pi^\dagger(c(\nabla)(\xi, \eta)) = \{\nabla(\xi),\nabla(\eta)\} - \nabla([\xi,\eta]) \ .
\]
The assignment $(\xi, \eta) \mapsto c(\nabla)(\xi, \eta)$ defines a map
\[
c(\nabla) \colon \mathcal{T}_X\otimes_\mathbb{C}\mathcal{T}_X \to \Omega^k_X
\]
called the \emph{curvature of $\nabla$}.

\begin{lemma}\label{lemma: curvature form}
{~}
\begin{enumerate}
\item The curvature $c(\nabla)$ of a connection $\nabla$ satisfies
\[
\iota_\eta (c(\nabla)(\xi_1, \xi_2)) = \langle \{\nabla(\xi_1),\nabla(\xi_2)\}, \nabla(\eta)\rangle
\]
for $\xi_1, \xi_2, \eta \in \mathcal{T}_X$
\item
The map
\begin{equation}\label{curvature form}
(\xi_1,\xi_2,\eta_1,\ldots,\eta_k) \mapsto \iota_{\eta_1\wedge\cdots\wedge\eta_k}(c(\nabla)(\xi_1, \xi_2))
\end{equation}
is totally skew-symmetric and $\mathcal{O}_X$-multilinear.

\item The differential form defined by \eqref{curvature form} is closed.
\end{enumerate}
\end{lemma}
\begin{proof}
The calculation
\begin{multline*}
\iota_\eta (c(\nabla)(\xi_1, \xi_2)) = \\
\langle c(\nabla)(\xi_1, \xi_2), \nabla(\eta)\rangle = \langle \{\nabla(\xi_1),\nabla(\xi_2)\}, \nabla(\eta)\rangle =
\\ L_{\xi_1} \langle \nabla(\xi_2), \nabla(\eta)\rangle - \langle \{\nabla(\xi_1),\nabla(\eta)\}, \nabla(\xi_2)\rangle = \\
- \langle \{\nabla(\xi_1),\nabla(\eta)\}, \nabla(\xi_2)\rangle =
- \iota_{\xi_2} (c(\nabla)(\xi_1, \eta))
\end{multline*}
proves the first claim. Since the map \eqref{curvature form} is
skew symmetric in the first two variables as well as in the rest
of the variables separately the calculation implies that
\eqref{curvature form} is, in fact, totally skew-symmetric. Since
it is $\mathcal{O}_X$-linear in $\eta_i$ and totally
skew-symmetric it follows that it is $\mathcal{O}_X$-linear in all
variables.

For any $\xi,\eta,\gamma\in\mathcal{T}_X$ we compute the left hand
side and the right hand side of the Jacobi identity,
\[
\{\nabla\xi,\{\nabla\eta,\nabla\gamma\}\}=\{\{\nabla\xi,\nabla\eta\},\nabla\gamma\}+\{\nabla\eta,\{\nabla\xi,\nabla\gamma\}\}.
\]
The left hand side is equal to:
\begin{multline}\label{curvature closed 1}
\{\nabla\xi,\{\nabla\eta,\nabla\gamma\}\}=\{\nabla\xi,
\pi^{\dagger}(c(\nabla)(\eta,\gamma))+\nabla[\eta,\gamma]\}\\
=\pi^{\dagger}(L_{\xi}c(\nabla)(\eta,\gamma))+\pi^{\dagger}(c(\nabla)(\xi,[\eta,\gamma]))+\nabla[\xi,[\eta,\gamma]] \\
=\pi^{\dagger}(L_{\xi}\iota_{\eta}\iota_{\gamma}c(\nabla)+\iota_{\xi}L_{\eta}\iota_{\gamma}c(\nabla)-\iota_{\xi}\iota_{\gamma}L_{\eta}c(\nabla))
+\nabla([\xi,[\eta,\gamma]]).
\end{multline}
Similarly, the right hand side is equal to:
\begin{multline}\label{curvature closed 2}
\{\{\nabla\xi,\nabla\eta\},\nabla\gamma\}+\{\nabla\eta,\{\nabla\xi,\nabla\gamma\}\}= \\
\{\pi^{\dagger}(c(\nabla)(\xi,\eta))+\nabla[\xi,\eta],\nabla\gamma\}+\{\nabla\eta,\pi^{\dagger}(c(\nabla)(\xi,\gamma))+\nabla[\xi,\eta]\} \\
=\pi^{\dagger}((L_{\xi}\iota_{\eta}\iota_{\gamma}-\iota_{\eta}L_{\xi}\iota_{\gamma}-\iota_{\gamma}d\iota_{\xi}\iota_{\eta}+L_{\eta}\iota_{\xi}\iota_{\gamma}
+\iota_{\eta}L_{\xi}\iota_{\gamma}-\iota_{\eta}\iota_{\gamma}L_{\xi})c(\nabla))\\
+\nabla([[\xi,\eta],\gamma]+[\eta,[\xi,\gamma]])
\end{multline}
Since $\eqref{curvature closed 1}$ is equal to $\eqref{curvature
closed 2}$, after cancellation of terms, the equality is written
as,
\begin{align*}
0 & =
\iota_{[\eta,\xi]}\iota_{\gamma}-\iota_{\gamma}L_{\xi}\iota_{\eta}-\iota_{\eta}\iota_{\gamma}L_{\xi}+\iota_{\gamma}\iota_{\xi}\iota_{\eta}dc(\nabla)
\\ & =
\iota_{\gamma}\iota_{[\eta,\xi]}+\iota_{[\eta,\xi]}\iota_{\gamma}+\iota_{\gamma}\iota_{\xi}\iota_{\eta}dc(\nabla)\\
& =\iota_{\gamma}\iota_{\xi}\iota_{\eta}dc(\nabla),
\end{align*}
\newline
thus $c(\nabla)$ is a closed differential form.
\end{proof}

In view of Lemma \ref{lemma: curvature form} we may and will
regard the curvature $c(\nabla)$ as a closed differential
$(k+2)$-form. Thus, the assignment $\nabla \mapsto c(\nabla)$
defines a map
\begin{equation}\label{curvature map}
c \colon \Conn(\mathcal{Q}) \to \Omega^{k+2,cl}_X \ .
\end{equation}

\begin{lemma}
The map \eqref{curvature map} satisfies $c(\nabla + \omega) =
c(\nabla) + d\omega$. Thus, the pair $(\Conn(\mathcal{Q}), c)$ is
a $(\Omega^{k+1}_X\xrightarrow{d}\Omega^{k+2,cl}_X)$-torsor.
\end{lemma}

The assignment $\mathcal{Q} \mapsto (\Conn(\mathcal{Q}), c)$ extends to a functor
\begin{equation}\label{ECA to torsors}
\ECA(X)_k \to \text{$(\Omega^{k+1}_X\xrightarrow{d}\Omega^{k+2,cl}_X)$-torsors}
\end{equation}

\begin{lemma}
The functor \eqref{ECA to torsors} is an equivalence.
\end{lemma}
\begin{proof}
We sketch a construction of a quasi-inverse. Suppose that
$(\mathcal{C},c)$ is a
$(\Omega^{k+1}_X\xrightarrow{d}\Omega^{k+2,cl}_X)$-torsor, i.e.
$\mathcal{C}$ is a $\Omega^{k+1}$-torsor and the map $c \colon
\mathcal{C} \to \Omega^{k+2,cl}_X$ satisfies $c(\nabla + \omega) =
c(\nabla) + d\omega$. We associate to it the exact Courant
algebroid which is the $\mathcal{C}$ twist of the Courant
algebroid $\mathcal{Q}_0$ of Example \ref{example: trivial ECA}
constructed as follows.

The underlying extension of $\mathcal{T}_X$ by $\Omega^k$ is the
$\mathcal{C}$ twist $\mathcal{Q}_0^\mathcal{C}$ of the trivial
extension $\mathcal{Q}_0 = \Omega^k\oplus\mathcal{T}_X$, i.e.
$\mathcal{Q}_0^\mathcal{C} =
\mathcal{C}\times_{\Omega^{k+1}_X}\mathcal{Q}_0$. Since the action
of $\Omega^{k+1}_X$ on $\mathcal{Q}_0$ preserves the symmetric
pairing, it follows that $\mathcal{Q}_0^\mathcal{C}$ is equipped
with the induced symmetric pairing. The Leibniz bracket on
$\mathcal{Q}_0^\mathcal{C}$ is given by the formula
\[
[(s_1,q_1),(s_2,q_2)] = (s_1,[q_1,q_2+\iota_{\pi(q_2)}(s_1-s_2)]_0 + \iota_{\pi(q_1)\wedge\pi(q_2)}c(s_1))\ ,
\]
where $s_i\in\mathcal{C}$, $q_i\in\mathcal{Q}_0$,
$s_1-s_2\in\Omega^{k+1}_X$ is the unique form such that $s_1 = s_2
+ (s_1-s_2)$ and $[\ ,\ ]_0$ is the bracket on $\mathcal{Q}_0$.
\end{proof}

The equivalence \eqref{ECA to torsors} induces a bijection between
the respective sets of connected components $\pi_0\ECA(X)_k \cong
H^1(X;\Omega^{k+1}_X\xrightarrow{d}\Omega^{k+2,cl}_X)$. 
The category $\ECA(X)_k$ has a canonical
structure of a Picard groupoid with the monoidal structure
(induced by the Baer sum of extensions), as does the category
$\text{$(\Omega^{k+1}_X\xrightarrow{d}\Omega^{k+2,cl}_X)$-torsors}$;
the functor \eqref{ECA to torsors} is monoidal.

\section{Transgression for Courant algebroids}\label{section: transgression for courant algebroids}

\subsection{The category of Courant-to-Lie morphisms}
Suppose that $\mathcal{Q}$ is a $(k-1)$-dimensional Courant
algebroid on $X$ (see Definition \ref{defn: courant}) and
$(\mathcal{A},\mathfrak{c})$ is marked Lie algebroid on $X^\sharp$
with $\deg\mathfrak{c} = k$ (see Definition \ref{defn: marked Lie
algd}).

\begin{definition}\label{def: CtL morphism}
A \emph{Courant-to-Lie (CtL) morphism $\mathcal{Q}
\xrightarrow{\phi} (\mathcal{A},\mathfrak{c})$} is a morphism of
$\mathcal{O}_X$-modules $\phi\colon\mathcal{Q}[1] \to \mathcal{A}$
such that the diagrams

\begin{multicols}{2}
\begin{equation}
\begin{CD}
\mathcal{Q} @>{\phi}>> \mathcal{A}^{-1} \\
@V{\pi}VV  @VV{\sigma}V \\
\mathcal{T}_X @>{\cong}>> \mathcal{T}_{X^\sharp}^{-1}
\end{CD}\label{ctl morphism anchor} \tag{A}
\end{equation}

\begin{equation}
\begin{CD}
\Omega_X^{k-1} @>{\cong}>> \left(\mathcal{O}_{X^\sharp}[k]\right)^{-1} \\
@V{\pi^\dagger}VV @VV{\cdot\mathfrak{c}}V \\
\mathcal{Q} @>{\phi}>> \mathcal{A}^{-1}
\end{CD}\label{ctl morphism coanchor} \tag{B}
\end{equation}

\begin{equation}
\begin{CD}
\mathcal{Q}\otimes_{\mathcal{O}_X}\mathcal{Q} @>{\phi\otimes\phi}>> \mathcal{A}^{-1}\otimes_{\mathcal{O}_X}\mathcal{A}^{-1} \\
@V{\langle\ ,\ \rangle}VV @VV{[\ ,\ ]^{-1,-1}}V \\
\Omega_X^{k-2} @>{\cdot\mathfrak{c}}>> \mathcal{A}^{-2}
\end{CD} \label{ctl morphism pairing} \tag{C}
\end{equation}

\begin{equation}
\begin{CD}
\mathcal{Q}\otimes_{\mathbb{C}}\mathcal{Q} @>{\phi\otimes\phi}>> \mathcal{A}^{-1}\otimes_{\mathbb{C}}\mathcal{A}^{-1} \\
@V{\{\ ,\ \}}VV @VV{[\delta(\ ) ,\ ]^{0,-1}}V \\
\mathcal{Q} @>{\phi}>> \mathcal{A}^{-1}
\end{CD}\label{ctl morphism bracket} \tag{D}
\end{equation}
\end{multicols}
are commutative. Note that since $\mathcal{O}_{X^\sharp}^i = 0$ for $i\leqslant -1$ the bracket $[\ ,\ ]^{-1,-1}$ (in the diagram \eqref{ctl morphism pairing}) is $\mathcal{O}_X$-bilinear (see Remark \ref{remark: bracket in negative degrees}).
\end{definition}

A morphism $(\mathcal{Q}_1 \xrightarrow{\phi_1}
(\mathcal{A}_1,\mathfrak{c}_1)) \xrightarrow{\Phi} (\mathcal{Q}_2
\xrightarrow{\phi_2} (\mathcal{A}_2,\mathfrak{c}_2))$ of
CtL-morphisms is a pair $\Phi = (\Phi_C,\Phi_L)$, where $\Phi_C
\colon \mathcal{Q}_1 \to \mathcal{Q}_2$ is a morphism of Courant
algebroids and $\Phi_L \colon (\mathcal{A}_1,\mathfrak{c}_1) \to
(\mathcal{A}_2,\mathfrak{c}_2)$ is a morphism of marked Lie
algebroids such that the diagram
\[
\begin{CD}
\mathcal{Q}_1[1] @>{\Phi_C[1]}>>  \mathcal{Q}_2[1] \\
@V{\phi_1}VV @VV{\phi_2}V \\
\mathcal{A}_1  @>{\Phi_L}>>  \mathcal{A}_2
\end{CD}
\]
is commutative. We leave it to the reader to check that with the above definitions CtL morphisms form a category.

\begin{notation}
We denote the category of Courant-to-Lie morphisms as above by $\mathtt{CtL}(X)_k$.
\end{notation}

The assignments $(\mathcal{Q} \xrightarrow{\phi} (\mathcal{A},\mathfrak{c})) \mapsto \mathcal{Q}$, $\Phi = (\Phi_C,\Phi_L) \mapsto \Phi_C$ define a functor
\begin{equation}\label{LA fibered over CA}
q\colon \mathtt{CtL}(X)_k \to \CA(X)_{k-1} \ .
\end{equation}

Suppose that $\psi_1\colon \mathcal{Q}_1 \to \mathcal{Q}_2$ is a
morphism in $\CA(X)_{k-1}$ and $(\mathcal{Q}_2
\xrightarrow{\phi_2} (\mathcal{A}_2,\mathfrak{c}_2)) \in
\mathtt{CtL}(X)_k$. We leave it to the reader to verify that the
composition $\mathcal{Q}_1[1] \xrightarrow{\psi_1[1]}
\mathcal{Q}_2[1] \xrightarrow{\phi_2[1]} \mathcal{A}_2)$ satisfies
the conditions of Definition \ref{def: CtL morphism}, i.e. defines
an object of $\mathtt{CtL}(X)_k$.

\begin{notation}
The object $(\mathcal{Q}_1 \xrightarrow{\phi_2\circ\psi_1} (\mathcal{A}_2,\mathfrak{c}_2))$ will be denoted $\psi_1^*(\mathcal{Q}_2 \xrightarrow{\phi_2} (\mathcal{A}_2,\mathfrak{c}_2))$.
\end{notation}

There is a canonical morphism
\begin{equation}\label{cartesian morphism}
\widetilde{\psi_1} \colon
\psi_1^*(\mathcal{Q}_2\xrightarrow{\phi_2}
(\mathcal{A}_2,\mathfrak{c}_2)) \to (\mathcal{Q}_2
\xrightarrow{\phi_2} (\mathcal{A},\mathfrak{c}_2))
\end{equation}
in $\mathtt{CtL}(X)$ given by $\widetilde{\psi_1} = (\psi_1, \id_\mathcal{A})$.

\begin{proposition}
{~}
\begin{enumerate}
\item The morphism \eqref{cartesian morphism} is cartesian.

\item The functor \eqref{LA fibered over CA} is a Grothendieck fibration.
\end{enumerate}
\end{proposition}
\begin{proof}
Suppose that $\Phi = (\Phi_C,\Phi_L) \colon (\mathcal{Q}
\xrightarrow{\phi} (\mathcal{A},\mathfrak{c})) \to (\mathcal{Q}_2
\xrightarrow{\phi} (\mathcal{A}_2,\mathfrak{c}_2))$ with $q(\Phi)
= \Phi_C = \psi_1\circ\psi$. Then $\Psi = (\psi, \Phi_L) \colon
(\mathcal{Q} \xrightarrow{\phi} (\mathcal{A},\mathfrak{c})) \to
\psi_1^*(\mathcal{Q}_2 \xrightarrow{\phi_2}
(\mathcal{A}_2,\mathfrak{c}_2))$ is the unique morphism such that
$\Phi = \widetilde{\psi_1}\circ\Psi$.

Since for any morphism $\psi_1\colon \mathcal{Q}_1 \to
\mathcal{Q}_2$ in $\CA(X)_{k-1}$ and any object $(\mathcal{Q}_2
\xrightarrow{\phi_2} (\mathcal{A}_2,\mathfrak{c}_2)) \in
\mathtt{CtL}(X)_k$ above $\mathcal{Q}_2$ there is a cartesian
morphism $\widetilde{\psi_1}$ with $q(\widetilde{\psi_1}) =
\psi_1$, it follows that the functor \eqref{LA fibered over CA} is
a fibration.
\end{proof}

For a $(k-1)$-dimensional Courant algebroid $\mathcal{Q}$ we consider the fiber $\mathtt{CtL}(X)_k/\mathcal{Q}$ of \eqref{LA fibered over CA} over $\mathcal{Q}$. Explicitly, the objects of $\mathtt{CtL}(X)_k/\mathcal{Q}$ are CtL morphisms $\mathcal{Q} \xrightarrow{\phi} (\mathcal{A},\mathfrak{c})$. A morphism $\Phi = (\Phi_C,\Phi_L)$ is in $\mathtt{CtL}(X)_k/\mathcal{Q}$ if $\Phi_C = \id_\mathcal{Q}$.

The rest of this section is devoted to the proof of the following theorem.

\begin{theorem}\label{thm: Q/CtL has initial object}
For any Courant algebroid $\mathcal{Q}$ the category $\mathtt{CtL}(X)_k/\mathcal{Q}$ has an initial object.
\end{theorem}

\begin{notation}
For a Courant algebroid $\mathcal{Q}$ we will denote the initial
object of $\mathtt{CtL}(X)_k/\mathcal{Q}$ by $\tau\mathcal{Q}$. By
abuse of notation we will refer to the marked Lie algebroid
constituent thereof by $\tau\mathcal{Q}$ as well.
\end{notation}

\begin{corollary}\label{cor: tau section}
The assignment $\mathcal{Q} \mapsto \tau\mathcal{Q}$ defines a section
\[
\tau \colon \CA(X)_{k-1} \longrightarrow \mathtt{CtL}(X)_k
\]
of \eqref{LA fibered over CA} (i.e. a functor which satisfies $q\circ\tau = \id$) which is left adjoint to $q$.
\end{corollary}
\begin{proof}
Suppose that $\psi \colon \mathcal{Q}_1 \to \mathcal{Q}_2$ is a
morphism in $\CA(X)_{k-1}$. Let $\tau(\psi) \colon
\tau\mathcal{Q}_1 \to \tau\mathcal{Q}_2$ denote the composition
$\tau\mathcal{Q}_1 \to \psi^*\tau\mathcal{Q}_2
\xrightarrow{\widetilde{\psi}} \tau\mathcal{Q}_2$. We leave it to
the reader to verify that the assignments $\mathcal{Q} \mapsto
\tau\mathcal{Q}$, $\psi \mapsto \tau(\psi)$ define a functor. It
is clear from the definition that $q\circ\tau = \id$.

For $\mathcal{Q}_1\in\CA(X)_{k-1}$ and $\phi = (\mathcal{Q}_2
\xrightarrow{\phi} (\mathcal{A}_2,\mathfrak{c}_2)) \in
\mathtt{CtL}(X)_k$ with $q(\phi) = \mathcal{Q}_2$ the map
\[
\Hom_{\mathtt{CtL}(X)_k}(\tau\mathcal{Q}_1, \phi) \to \Hom_{\CA(X)_{k-1}}(\mathcal{Q}_1, q(\phi))
\]
is defined by $\Phi = (\Phi_C,\Phi_L) \mapsto \Phi_C$. This map is
clearly natural (i.e. a morphism of bi-functors) and is, in fact
an isomorphism of such with the inverse given by $\psi \mapsto
(\psi, \Psi)$ where $\Psi$ is the composition $\tau\mathcal{Q}_1
\to \psi^*\tau\mathcal{Q}_2 \xrightarrow{\widetilde{\psi}}
\tau\mathcal{Q}_2 \to \mathcal{A}_2$.
\end{proof}

\subsection{Construction of a marked Lie algebroid}

In this section we are going to construct a marked Lie algebroid
that we will denote by $(\tau\mathcal{Q},\mathfrak{c})$. Let
$\tau\mathcal{Q}$ denote the cokernel of the map
\[
\tau\mathcal{Q} := \coker(\pr_*\ev^*\Omega_X^{k-1}
\xrightarrow{(\int,-\pr_*\ev^*(\pi^\dagger))}
\mathcal{O}_{X^\sharp}[k]\oplus \pr_*\ev^*\mathcal{Q}) \ ,
\]
where $\int$ is the map \eqref{integration map}. In other words, the square
\begin{equation}\label{defn tauQ coCart square}
\begin{CD}
\pr_*\ev^*\Omega_X^{k-1} @>{\pr_*\ev^*(\pi^\dagger)}>> \pr_*\ev^*\mathcal{Q} \\
@V{\int}VV @VVV \\
\mathcal{O}_{X^\sharp}[k] @>>> \tau\mathcal{Q}
\end{CD}
\end{equation}
is cocartesian.

\subsubsection{The bracket on $\tau\mathcal{Q}$.}\label{subsubsection: the bracket on tauQ}
Let
\begin{multline}\label{bracket no scalars}
[\ ,\ ] \colon (\mathcal{O}_{X^\sharp}[k]\oplus
\mathcal{O}_{X^\sharp}[\epsilon]\otimes_\mathbb{C}\mathcal{Q})\otimes_\mathbb{C}(\mathcal{O}_{X^\sharp}[k]\oplus
\mathcal{O}_{X^\sharp}[\epsilon]\otimes_\mathbb{C}\mathcal{Q}) \to
\\
\mathcal{O}_{X^\sharp}[k]\oplus
\mathcal{O}_{X^\sharp}[\epsilon]\otimes_\mathbb{C}\mathcal{Q}
\end{multline}
denote the extension by Leibniz rule of the binary operation on
$\mathcal{O}_{X^\sharp}[k]\oplus
\mathcal{O}_{\vec{\mathfrak{t}}}\otimes_\mathbb{C}\mathcal{Q}$
determined by
\begin{enumerate}
\item $[\mathcal{O}_{X^\sharp}[k],\mathcal{O}_{X^\sharp}[k]] = 0$

\item $[\epsilon\otimes q, \beta] = -(-1)^{\beta+k}[\beta,
\epsilon\otimes q] = \iota_{\pi(q)}\beta$

\item $[1\otimes q,
\beta] = -[\beta, 1\otimes q] = L_{\pi(q)}\beta$

\item
$[\epsilon\otimes q_1,\epsilon\otimes q_2] = \langle
q_1,q_2\rangle \in (\mathcal{O}_{X^\sharp}[k])^{-2}$

\item
$[1\otimes q_1,\epsilon\otimes q_2] = \epsilon\otimes\{q_1,q_2\}$

\item $[\epsilon\otimes q_1,1\otimes q_2] = -d\langle
q_1,q_2\rangle + \epsilon\otimes\{q_1,q_2\}$
\end{enumerate}
where $q,q_1,q_2\in\mathcal{Q}$ and
$\beta\in\mathcal{O}_{X^\sharp}$.
\begin{notation}
From now on we will denote $\alpha\otimes1\otimes q$ (respectively
$\alpha\otimes\epsilon\otimes q$) by $\alpha\otimes q$
(respectively $\alpha\epsilon\otimes q$)
\end{notation}
A general element of $\mathcal{O}_{X^\sharp}[k]\oplus
\mathcal{O}_{X^\sharp}[\epsilon]\otimes_\mathbb{C}\mathcal{Q}$ is
sum of ones of the form $\omega$, $\alpha\epsilon\otimes q$ and
$\beta\otimes r$, where
$\omega,\alpha,\beta\in\mathcal{O}_{X^{\sharp}}$ and
$q,r\in\mathcal{Q}$. The bracket \eqref{bracket no scalars} of two
such is given explicitly by formulas formulas \eqref{bracket with
scalars-1}--\eqref{bracket with scalars-6} below, where
$\alpha,\alpha_1,\alpha_2,\beta\in\mathcal{O}_{X^{\sharp}}$ are
homogeneous elements and $q_1,q_2\in\mathcal{Q}$.
\begin{multline}\label{bracket with scalars-1}
[\alpha_1\epsilon\otimes q_1, \alpha_2\epsilon\otimes
q_2] = \\
\alpha_1\wedge(\iota_{\pi(q_1)}\alpha_2)\epsilon\otimes
q_2+(-1)^{(\alpha_1-1)\alpha_2}\alpha_2\wedge\alpha_1\wedge\langle
q_1,q_2\rangle\\
+(-1)^{(\alpha_1-1)\alpha_2+\alpha_1}\alpha_1\wedge(\iota_{\pi(q_2)}\alpha_1)\epsilon\otimes
q_1
\end{multline}

\begin{multline}\label{bracket with scalars-2}
[\alpha_1\otimes q_1, \alpha_2\epsilon\otimes q_2]= \\
\alpha_1\wedge(L_{\pi(q_1)}\alpha_2)\epsilon\otimes
q_2+(-1)^{\alpha_1\alpha_2}\alpha_2\wedge\alpha_2\epsilon\otimes\{q_1,q_2\}\\
-(-1)^{\alpha_1\alpha_2+\alpha_1}\alpha_2\wedge(\iota_{\pi(q_2)}\alpha_1)\otimes
q_1.
\end{multline}

\begin{multline}\label{bracket with scalars-3}
[\alpha_1\epsilon\otimes q_1, \alpha_2\otimes
q_2] = \\
\alpha_1\wedge(\iota_{\pi(q_1)}\alpha_2)\otimes
q_2-(-1)^{(\alpha_1-1)\alpha_2}\alpha_2\wedge\alpha_1\wedge
d\langle q_1,q_2\rangle\\
+
(-1)^{(\alpha_1-1)\alpha_2}\alpha\wedge\alpha_1\epsilon\otimes\{q_1,q_2\}
-
(-1)^{(\alpha_1-1)\alpha_2}\alpha_2\wedge(L_{\pi(q_2)}\alpha_1)\epsilon\otimes
q_1.
\end{multline}

\begin{multline}\label{bracket with scalars-4}
[\alpha_1\otimes q_1, \alpha_2\otimes
q_2] = \\
\alpha_1\wedge(L_{\pi(q_1)}\alpha_2)\otimes
q_2+(-1)^{\alpha_1\alpha_2}\alpha_2\wedge\alpha_1\otimes\{q_1,q_2\}\\
-(-1)^{\alpha_1\alpha_2}\alpha_2\wedge(L_{\pi(q_2)}\alpha_1)\otimes
q_1.
\end{multline}

\begin{equation}\label{bracket with scalars-5}
[\alpha\epsilon\otimes q,\beta]=-(-1)^{\alpha\beta+\beta}[\beta,
\alpha\epsilon\otimes q]=\alpha\wedge\iota_{\pi(q)}\beta
\end{equation}

\begin{equation}\label{bracket with scalars-6}
[\alpha\otimes q,\beta]=-(-1)^{\alpha\beta}[\beta,\alpha\otimes
q]=\alpha\wedge L_{\pi(q)}\beta .
\end{equation}

\begin{lemma} \label{lemma: bracket Leibniz}
The operation \eqref{bracket no scalars} satisfies
\begin{equation}\label{bracket and differential}
\partial([a,b])=[\partial(a),b]+(-1)^{a}[a,\partial(b)],
\end{equation}
for any pair of homogeneous elements
$a,b\in\mathcal{O}_{X^\sharp}[k]\oplus
\mathcal{O}_{X^\sharp}[\epsilon]\otimes_\mathbb{C}\mathcal{Q}$.
\end{lemma}
\begin{proof}
The proof is given in \ref{subsection: proof bracket Leibniz}
\end{proof}

Let
\[
K := \ker(\mathcal{O}_{X^\sharp}[k]\oplus
\mathcal{O}_{X^\sharp}[\epsilon]\otimes_\mathbb{C}\mathcal{Q} \to
\tau\mathcal{Q}) .
\]
It is equal to the $\mathcal{O}_{X^\sharp}$-submodule of
$\mathcal{O}_{X^\sharp}[k]\oplus
\mathcal{O}_{X^\sharp}[\epsilon]\otimes_\mathbb{C}\mathcal{Q}$
generated by
\begin{enumerate}
\item $\alpha - \epsilon\otimes\pi^\dagger(\alpha)$ for
$\alpha\in\Omega^{k-1}_X$

\item $d\alpha - 1\otimes\pi^\dagger(\alpha)$ for
$\alpha\in\Omega^{k-1}_X$

\item $\beta\wedge\alpha -
\beta\epsilon\otimes\pi^\dagger(\alpha)$ for
$\beta\in\mathcal{O}_{X^\sharp}$ and $\alpha\in\Omega^{k-1}_X$

\item $\epsilon\otimes fq - f\epsilon\otimes q$ for
$f\in\mathcal{O}_X$ and $q \in \mathcal{Q}$

\item $1\otimes fq - f\otimes q - df\epsilon\otimes q$ for
$f\in\mathcal{O}_X$ and $q \in \mathcal{Q}$
\end{enumerate}
Note that $\partial(K) \subset K$.

\begin{lemma}\label{lemma: two-sided ideal}
$K$ is a two-sided ideal with respect to the operation
\eqref{bracket no scalars}, i.e.
\[
[\mathcal{O}_{X^\sharp}[k]\oplus
\mathcal{O}_{X^\sharp}[\epsilon]\otimes_\mathbb{C}\mathcal{Q}, K]
\subset K,\ \ \ [K, \mathcal{O}_{X^\sharp}[k]\oplus
\mathcal{O}_{X^\sharp}[\epsilon]\otimes_\mathbb{C}\mathcal{Q}]
\subset K \ .
\]
\end{lemma}
\begin{proof}
The proof is given in \ref{subsection: two-sided ideal}
\end{proof}

Lemma \ref{lemma: two-sided ideal} implies that the operation \eqref{bracket no scalars} descends to the operation
\begin{equation}\label{bracket on tau}
[\ ,\ ] \colon \tau\mathcal{Q}\otimes_\mathbb{C}\tau\mathcal{Q}
\to \tau\mathcal{Q} \ .
\end{equation}

\subsubsection{The anchor map and the marking of $\tau\mathcal{Q}$.}
The composition
\begin{multline*}
\pr_*\ev^*\Omega^{k-1}
\xrightarrow{(\int,-\pr_*\ev^*(\pi^\dagger))}
\mathcal{O}_{X^\sharp}[k]\oplus \pr_*\ev^*\mathcal{Q} \to
\\
\pr_*\ev^*\mathcal{Q} \xrightarrow{\pr_*\ev^*(\pi)}
\pr_*\ev^*\mathcal{T}_X = \mathcal{T}_{X^\sharp}
\end{multline*}
is equal to zero, hence factors through $\tau\mathcal{Q}$. Let
\begin{equation}\label{anchor on tau}
\sigma \colon \tau\mathcal{Q} \to \mathcal{T}_{X^\sharp}
\end{equation}
denote the induced map of $\mathcal{O}_{X^{\sharp}}$-modules.

\begin{lemma}
The map \eqref{anchor on tau} preserves the brackets, i.e.
\[
\sigma([a,b]) = [\sigma(a), \sigma(b)],
\]
for any $a,b\in\tau\mathcal{Q}$.
\end{lemma}
\begin{proof}
Left to the reader.
\end{proof}

Let $\mathfrak{c}\in\tau\mathcal{Q}^{-k}$ denote the image of
$1\in\Gamma(X;(\mathcal{O}_{X^\sharp}[k])^{-k})$ under the
composition
\[
\mathcal{O}_{X^\sharp}[k] \xrightarrow{(\id,0)} \mathcal{O}_{X^\sharp}[k]\oplus
\pr_*\ev^*\mathcal{Q} \to \tau\mathcal{Q}
\]
(equal to the bottom horizontal map in \eqref{defn tauQ coCart square}).

\begin{proposition}\label{tau as a marked Lie}
The operation \eqref{bracket on tau}, the anchor map
\eqref{anchor on tau} and the marking $\mathfrak{c}$ endow the
$\mathcal{O}_{X^\sharp}$-module $\tau\mathcal{Q}$ with a structure
of a marked $\mathcal{O}_{X^\sharp}$-Lie algebroid.
\end{proposition}
\begin{proof}
The proof is given in \ref{subsection: tau as a marked Lie}.
\end{proof}

Since the diagram \eqref{defn tauQ coCart square} is commutative, the (right vertical) map $\pr_*\ev^*\mathcal{Q} \to \tau\mathcal{Q}$ induces the map
\begin{equation}\label{coker to coker}
\coker(\pr_*\ev^*(\pi^\dagger)) \to \coker(\mathcal{O}_{X^\sharp}[k] \xrightarrow{\cdot\mathfrak{c}} \tau\mathcal{Q}) = \overline{\tau\mathcal{Q}} .
\end{equation}

\begin{proposition}\label{prop: properties of transgression}
The marked $\mathcal{O}_{X^\sharp}$-Lie algebroid $\tau\mathcal{Q}$ enjoys the following properties:
\begin{enumerate}
\item $\coker(\pr_*\ev^*(\pi^\dagger)) = \overline{\mathcal{Q}}^\sharp$

\item the map \eqref{coker to coker} is an isomorphism, in particular $\overline{\mathcal{Q}}^\sharp \cong \overline{\tau\mathcal{Q}}$ and
$\overline{\tau\mathcal{Q}}^i = 0$ for $i \leqslant -2$;

\item the map $\Omega^{i+k}_X \to \tau\mathcal{Q}^i$ induced by the map
$\mathcal{O}_{X^\sharp}[k] \xrightarrow{\cdot\mathfrak{c}} \tau\mathcal{Q}$ is an isomorphism for $i \leqslant -2$;

\item if $\mathcal{Q}$ is a $\Omega^{k-1}_X$-extension of
$\mathcal{A}$ then $\tau\mathcal{Q}$ is a
$\mathcal{O}_{X^\sharp}[k]$-extension of $\mathcal{A}^\sharp$.
\end{enumerate}
\end{proposition}
\begin{proof}
The map \eqref{coker to coker} is an isomorphism since the square
\eqref{defn tauQ coCart square} is cocartesian. Since the functor
$\pr_*\ev^*$ is exact, it follows that
$\coker(\pr_*\ev^*(\pi^\dagger)) = \overline{\mathcal{Q}}^\sharp$.
We leave it to the reader to check that the map
$\overline{\mathcal{Q}}^\sharp \to \overline{\tau\mathcal{Q}}$ is
a morphism of $\mathcal{O}_{X^\sharp}$-Lie algebroids.

The second claim is a direct consequence of the construction of $\tau\mathcal{Q}$.

If the sequence \eqref{CExt transitive exact sequence} is exact
then so is the sequence obtained from \eqref{CExt transitive exact
sequence} by applying the exact functor $\pr_*\ev^*$. The latter
short exact sequence gives rise, by push-out along the map
$\int\colon \pr_*\ev^*\Omega_X^{k-1} \to
\mathcal{O}_{X^\sharp}[k]$ (see \eqref{defn tauQ coCart square}),
to the short exact sequence
\[
0 \to \mathcal{O}_{X^\sharp}[k] \xrightarrow{\cdot\mathfrak{c}} \tau\mathcal{Q} \to \mathcal{A}^\sharp \to 0 \ .
\]
\end{proof}

\subsection{The initial CtL morphism}
\begin{lemma}
The composition
\begin{equation}\label{init CtL}
\mathcal{Q}[1] \to \pr_*\ev^*\mathcal{Q} \to \mathcal{O}_{X^\sharp}[k]\oplus
\pr_*\ev^*\mathcal{Q} \to \tau\mathcal{Q}
\end{equation}
defines a CtL morphism.
\end{lemma}
\begin{proof}
We need to show that the squares \eqref{ctl morphism anchor},
\eqref{ctl morphism coanchor}, \eqref{ctl morphism pairing} and
\eqref{ctl morphism bracket} in Definition \ref{def: CtL morphism}
commute.

Square \eqref{ctl morphism anchor} commutes since the two
compositions $q\mapsto\epsilon\otimes q\mapsto\iota_{\pi(q)}$ and
$q\mapsto\pi(q)\mapsto\iota_{\pi(q)}$ are equal.

Square \eqref{ctl morphism coanchor} commutes since the two
compositions are $\alpha \mapsto \alpha\cdot\mathfrak{c}$ and
$\alpha \mapsto \epsilon\otimes\pi^{\dagger}(\alpha)$ and
$\alpha\cdot\mathfrak{c} = \epsilon\otimes\pi^{\dagger}(\alpha)$
in $\tau\mathcal{Q}$.

Square \eqref{ctl morphism pairing} commutes since
$[\epsilon\otimes q_1,\epsilon\otimes q_2] = \langle
q_1,q_2\rangle\cdot\mathfrak{c}$ by definition of the bracket on
$\tau\mathcal{Q}$ (see \ref{subsubsection: the bracket on tauQ}).

Square \eqref{ctl morphism bracket} commutes since
$\delta(\epsilon\otimes q_1) = 1\otimes q_1$ and $[1\otimes
q_1,\epsilon\otimes q_2] = \epsilon\otimes\{q_1,q_2\}$ by
definition of the bracket on $\tau\mathcal{Q}$ (see
\ref{subsubsection: the bracket on tauQ}).
\end{proof}

\begin{proposition} The object $(\mathcal{Q} \to (\tau\mathcal{Q},\mathfrak{c}))\in\mathcal{Q}/\mathtt{CtL}$ corresponding to the CtL morphism \eqref{init CtL} is initial.
\end{proposition}
\begin{proof}
Suppose that $\mathcal{Q} \xrightarrow{\phi}
(\mathcal{A},\mathfrak{c}_{\mathcal{A}})$ is a CtL morphism. The morphism
$\phi\colon\mathcal{Q}[1] \to \mathcal{A}$ admits a unique
extension to a morphism of $\mathcal{O}_{X^\sharp}$-modules
$\phi \colon \pr_*\ev^*\mathcal{Q} \to \mathcal{A} $ given by the formula
\[
\phi(\omega\otimes q_1+\beta\epsilon\otimes
q_2)=\omega\cdot\delta(\phi(q_1))+\beta\cdot\phi(q_2),
\]
where $\omega,\beta\in\mathcal{O}_{X^{\sharp}}$ and $q_1,q_2\in\mathcal{Q}$. Since the diagram
\[
\begin{CD}
\pr_*\ev^*\Omega^{k-1} @>{\pr_*\ev^*(\pi^\dagger)}>> \pr_*\ev^*\mathcal{Q} \\
@V{\int}VV @VV{\phi}V \\
\mathcal{O}_{X^\sharp}[k] @>{\cdot\mathfrak{c}}>> \mathcal{A}
\end{CD}
\]
is commutative, there exists a unique morphism of $\mathcal{O}_{X^\sharp}$-modules
\begin{equation}\label{tau to lie}
\widetilde{\phi} \colon \tau\mathcal{Q} \to \mathcal{A}
\end{equation}
It is given by
\[
\widetilde{\phi}(\theta+\omega\epsilon\otimes q_1+\beta\otimes
q_2) =
\theta\cdot\mathfrak{c}_{\mathcal{A}}+\omega\cdot\phi(q_1)+\beta\cdot\delta(\phi(q_2)).
\]
We leave it to the reader to verify that \eqref{tau to lie} is a morphism of marked Lie algebroids. Since, clearly, the diagram
\[
\begin{CD}
\mathcal{Q}[1] @= \mathcal{Q}[1] \\
@V{\eqref{init CtL}}VV @VV{\phi}V \\
\tau\mathcal{Q} @>{\widetilde{\phi}}>> \mathcal{A}
\end{CD}
\]
is commutative, \eqref{tau to lie} is a morphism in $\mathcal{Q}/\mathtt{CtL}$.
\end{proof}

\section{Trasgression for extensions}\label{section: transgression for extensions}

\subsection{From Lie to Courant}\label{From Lie to Courant}
Suppose that $(\mathcal{B},\mathfrak{c}) \in \LA{\mathcal{O}_{X^\sharp}}^\star_n$ satisfies
\begin{enumerate}
\item the sequence
\[
0 \to \mathcal{O}[n] \xrightarrow{\cdot\mathfrak{c}} \mathcal{B} \to \overline{\mathcal{B}} \to 0
\]
is exact, i.e.$(\mathcal{B},\mathfrak{c})$ is an $\mathcal{O}[n]$-extension of $\overline{\mathcal{B}}$;

\item $\overline{\mathcal{B}}^i = 0$ for $i \leqslant -2$, i.e.
the map $\mathcal{O}[n] \xrightarrow{\cdot\mathfrak{c}}
\mathcal{B}$ induces isomorphisms $\Omega^{i+n}_X
\xrightarrow{\cong} \mathcal{B}^i$ for $i \leqslant -2$.
\end{enumerate}

Let $\Cour(\mathcal{B},\mathfrak{c})$ denote the $\mathcal{O}_X$-module $\mathcal{B}^{-1}$ equipped with the following structures:
\begin{enumerate}
\item the map $\pi \colon \mathcal{B}^{-1} \to
\mathcal{T}_{\mathcal{O}_{X^\sharp}}^{-1} = \mathcal{T}_X$ is the
restriction of the anchor map $\mathcal{B} \to
\mathcal{T}_{\mathcal{O}_{X^\sharp}}$;

\item the map $\pi^\dagger \colon \Omega^{n-1} \to
\mathcal{B}^{-1}$ is the restriction of the map $\mathcal{O}[n]
\xrightarrow{\cdot\mathfrak{c}} \mathcal{B}$;

\item the pairing $\langle\ ,\ \rangle \colon
\mathcal{B}^{-1}\otimes\mathcal{B}^{-1} \to \Omega^{n-2}$ is
determined by the equation
\[
\langle a, b\rangle\cdot\mathfrak{c} = [a,b]^{-1,-1} \ ;
\]

\item the binary operation $\{\ ,\ \} \colon
\mathcal{B}^{-1}\otimes\mathcal{B}^{-1} \to \mathcal{B}^{-1}$ is
the derived bracket, i.e. is given by the formula
\[
\{a,b\} = [\delta a, b]^{0,-1} \ .
\]
\end{enumerate}

\begin{lemma}
{~}
\begin{enumerate}
\item $\Cour(\mathcal{B},\mathfrak{c})$ is an $(n-1)$-dimensional Courant algebroid.

\item $\overline{\Cour(\mathcal{B},\mathfrak{c})} =
\overline{\mathcal{B}}^{-1}$; in particular, the derived bracket
defines a structure of a Lie algebroid on
$\overline{\mathcal{B}}^{-1}$ and
$\Cour(\mathcal{B},\mathfrak{c})$ is a Courant extension of
$\overline{\mathcal{B}}^{-1}$.

\item The inclusion $\mathcal{B}^{-1}[1] \to \mathcal{B}$ defines
a CtL morphism $\Cour(\mathcal{B},\mathfrak{c}) \to
(\mathcal{B},\mathfrak{c})$.
\end{enumerate}
\end{lemma}
\begin{proof}
We leave it to the reader to verify that $\Cour(\mathcal{B},\mathfrak{c})$ is a Courant algebroid.
Note that the sequence
\[
0 \to \Omega^{n-1}_X \to \mathcal{B}^{-1} \to \overline{\mathcal{B}}^{-1} \to 0
\]
is exact by assumption.

The last claim is a direct consequence of the definition of Courant algebroid structure on $\Cour(\mathcal{B},\mathfrak{c})$.
\end{proof}

The assignment $(\mathcal{B},\mathfrak{c}) \mapsto (\Cour(\mathcal{B},\mathfrak{c}) \to (\mathcal{B},\mathfrak{c}))$ extends to a functor
\begin{equation}\label{OExt to CtL}
\OExt{n}(\overline{\mathcal{B}}) \longrightarrow \Cour(\mathcal{B},\mathfrak{c})/\mathtt{CtL} .
\end{equation}
Composing with the forgetful functor $\Cour(\mathcal{B},\mathfrak{c})/\mathtt{CtL} \to \CExt{\overline{\mathcal{B}}^{-1}}{n-1}$ we obtain the functor
\[
\Cour \colon  \OExt{n}(\overline{\mathcal{B}}) \longrightarrow \CExt{\overline{\mathcal{B}}^{-1}}{n-1} .
\]

\subsection{From Courant to Lie}\label{From Courant to Lie}
Suppose that $\mathcal{A}$ is an $\mathcal{O}_X$-Lie algebroid.
According to \ref{subsection: Transgression for Lie algebroids}
$\mathcal{A}$ gives rise to the $\mathcal{O}_{X^\sharp}$-Lie
algebroid $\mathcal{A}^\sharp$ with
$\left(\mathcal{A}^\sharp\right)^{-1} = \mathcal{A}$. Then, any
$\mathcal{B} \in \OExt{n}(\mathcal{A}^\sharp)$ satisfies the
assumptions of \ref{From Lie to Courant}, whence the functor
\begin{equation}\label{cour ext to ext}
\Cour \colon \OExt{n}(\mathcal{A}^\sharp) \to \CExt{\mathcal{A}}{n-1} .
\end{equation}

On the other hand, according to Proposition \ref{prop: properties of transgression}, the transgression functor restricts to the functor
\begin{equation}\label{tau ext to ext}
\tau \colon \CExt{\mathcal{A}}{n-1} \to \OExt{n}(\mathcal{A}^\sharp) .
\end{equation}

\begin{theorem}\label{thm: ext to ext equiv}
The functors \eqref{cour ext to ext} and \eqref{tau ext to ext} are mutually quasi-inverse equivalences of categories.
\end{theorem}

\begin{proof}
It is clear that $\Cour\circ\tau=\id$.

For $(\mathcal{B},\mathfrak{c}) \in \OExt{n}(\mathcal{A}^\sharp)$
the CtL morphism $\Cour(\mathcal{B},\mathfrak{c}) \to
(\mathcal{B},\mathfrak{c})$ gives rise to the morphism
$\tau\Cour(\mathcal{B},\mathfrak{c}) \to \mathcal{B}$ of
$\mathcal{O}_X$-Lie algebroids and, in fact, of
$\mathcal{O}_X[n]$-extensions of $\mathcal{A}^\sharp$. Therefore,
it is an isomorphism. It is clearly natural in
$(\mathcal{B},\mathfrak{c})$, hence $\tau\circ\Cour\cong id$.
\end{proof}

\subsection{Transgression and symplectic NQ-manifolds of degree 2}
Below we sketch the relationship between the transgression functor and the construction of a NQ-manifold of degree 2 associated to a (1-dimensional) Courant algebroid of \cite{R} and attributed to A.~Weinstein in \cite{S3}.

Suppose that $\mathcal{E}$ is a vector bundle on $X$ and $\ip \colon \mathcal{E}\otimes_{\mathcal{O}_X}\mathcal{E} \to \mathcal{O}_X$ is a non-degenerate symmetric pairing. Let $\mathfrak{so}(\mathcal{E},\ip)$ denote the subsheaf of endomorphisms $\phi\in\shEnd_{\mathcal{O}_X}(\mathcal{E})$ such that $\langle \phi(e), f\rangle + \langle e, \phi(f) \rangle = 0$ for all $e,f \in \mathcal{E}$. Let $\theta \colon \textstyle{\bigwedge}^2\mathcal{E} \to \mathfrak{so}(\mathcal{E},\ip)$ denote the map defined by $\theta(e_1\wedge e_2)(f) = \dfrac12\left(\langle e_2,f\rangle e_1 - \langle e_1,f\rangle e_2\right)$. The map $\theta$ is an isomorphism since the pairing $\ip$ is non-degenerate.

Let $\mathcal{A}_{\mathcal{E},\ip}$ denote the subsheaf of operators $D \in \mathcal{A}_\mathcal{E}$ which satisfy $\langle D(e), f\rangle + \langle e, D(f) \rangle = \sigma(D)(\langle e, f\rangle)$. The sheaf $\mathcal{A}_{\mathcal{E},\ip}$ is a $\mathcal{O}_X$-algebroid. The anchor map $\sigma\colon \mathcal{A}_{\mathcal{E},\ip} \to \mathcal{T}_X$ is surjective with kernel $\mathfrak{so}(\mathcal{E},\ip)$.

Let $I \subset \Symm_{\mathcal{O}_X}(\mathcal{E}[-1]\oplus\mathcal{A}_{\mathcal{E},\ip}[-2])$ denote the ideal generated by the sections of the form $\theta(g) - g \in S_{\mathcal{O}_X}(\mathcal{E}[-1]\oplus\mathcal{A}_{\mathcal{E},\ip}[-2])^2$, $g\in\bigwedge^2\mathcal{E}$. Let
\[
e_3(\mathcal{E},\ip) := S_{\mathcal{O}_X}(\mathcal{E}[-1]\oplus\mathcal{A}_{\mathcal{E},\ip}[-2])/I .
\]
Since the relations are homogeneous $e_3(\mathcal{E},\ip)$ inherits a structure of a graded commutative $\mathcal{O}_X$-algebra. Moreover, $e_3(\mathcal{E},\ip)^0 = \mathcal{O}_X$, $e_3(\mathcal{E},\ip)^1 = \mathcal{E}$, $e_3(\mathcal{E},\ip)^2 = \mathcal{A}_{\mathcal{E},\ip}$.

The bracket on $\mathcal{A}_{\mathcal{E},\ip}$, the action of the latter on $\mathcal{E}$ and the paring $\ip$ extend canonically to a structure of a Lie algebra on $S_{\mathcal{O}_X}(\mathcal{E}[-1]\oplus\mathcal{A}_{\mathcal{E},\ip}[-2])[2]$ so that the bracket is a bi-derivation of the product. Since $\{S_{\mathcal{O}_X}(\mathcal{E}[-1]\oplus\mathcal{A}_{\mathcal{E},\ip}[-2])[2], I\} \subset I$, the bracket descends to $e_3(\mathcal{E},\ip)[2]$. The commutative algebra $e_3(\mathcal{E},\ip)$ equipped with the bracket $\{\ ,\ \}$ is a $e_3$-algebra.

A structure on a Courant algebroid on $\mathcal{E}$ with the symmetric pairing given by $\ip$ gives rise to a derivation of the $e_3$-algebra $e_3(\mathcal{E},\ip)$ of degree one and square zero. As a consequence, there is a canonical map of DGA $\mathcal{O}_{X^\sharp} \to e_3(\mathcal{E},\ip)$ which extends the identity map of $\mathcal{O}_X$ in degree zero.

Suppose that $\mathcal{Q}$ is a Courant algebroid. Thus, $e_3(\mathcal{Q},\ip)$ is a differential $e_3$-algebra equipped with the morphism of DGA $\pi^\dagger \colon \mathcal{O}_{X^\sharp} \to (e_3(\mathcal{Q},\ip)$ \emph{which we assume to be a monomorphism}. Let $F_1 e_3(\mathcal{Q},\ip)[2]$ denote the normalizer of $\im(\pi^\dagger)$. The sheaf $F_1 e_3(\mathcal{Q},\ip)[2]$ has a canonical structure of a $\mathcal{O}_{X^\sharp}$-Lie algebroid marked by $\pi^\dagger(1)$. Note that $\mathcal{Q} = F_1 e_3(\mathcal{Q},\ip)^1 = e_3(\mathcal{Q},\ip)^1$ and $\mathcal{A}_{\mathcal{Q},\ip} \subset F_1 e_3(\mathcal{Q},\ip)^2$.

The map $\mathcal{Q}[1] \to F_1 e_3(\mathcal{Q},\ip)[2]$ is a CtL morphism $\mathcal{Q} \to (F_1 e_3(\mathcal{Q},\ip)[2], \pi^\dagger(1))$ hence extends to a morphism of marked $\mathcal{O}_{X^\sharp}$-Lie algebroids
\[
(\tau\mathcal{Q},\mathfrak{c}) \to (F_1 e_3(\mathcal{Q},\ip)[2], \pi^\dagger(1)) .
\]

\subsection{Some examples}
\begin{example}{\emph{(Exact Courant algebroids)}}
Applying Theorem \ref{thm: ext to ext equiv} to $\mathcal{T}_X$ we obtain the equivalence of \cite{CR}, namely the equivalence of categories $\ECA(X)_{n-1} \cong \OExt{n}(\mathcal{T}_{X^\sharp})$ (denoted $\mathcal{O}_{X^\sharp}-\operatorname{LA}(X^\sharp)$ in loc. cit.).
\end{example}

\begin{example}{\emph{(Commutative Courant algebroids)}}
The sheaf of differential forms of degree $n-1$ has a canonical structure of a commutative $(n-1)$-dimensional Courant algebroid. The $\mathcal{O}_{X^\sharp}$-Lie algebroid $\tau\Omega^n_X$ is $\mathcal{O}_{X^\sharp}[n]$ with trivial bracket and anchor.
\end{example}

\begin{example}{\emph{(Quadratic Lie algebras)}}
Suppose that $\mathcal{A}$ is a transitive $\mathcal{O}_X$-Lie algebroid. Let $\mathfrak{g}$ denote the kernel of the anchor map. Let $\ip$ be a $\mathcal{A}$-invariant symmetric $\mathcal{O}_X$-bilinear pairing on $\mathfrak{g}$. This data determines a Courant algebroid with trivial anchor $\widehat{\mathfrak{g}}$ which is a Courant extension of $\mathfrak{g}$ (see \cite{B}, 3.2 for details).
The Courant algebroid $\widehat{\mathfrak{g}}$ transgresses to an abelian extension of the DGLA $\mathfrak{g}^\sharp$ by $\mathcal{O}_{X^\sharp}[2]$.

The construction applies when $X$ is a connected manifold of dimension zero, in which case $\mathfrak{g}$ is a $\mathbb{C}$-Lie algebra equipped with an invariant symmetric bilinear form. In this case the DGLA $\tau\mathfrak{g}$ is easily seen to be $\mathbb{C}[2]\oplus\mathfrak{g}[1]\oplus\mathfrak{g}$ with the bracket given by the symmetric form, the adjoint action and the bracket on $\mathfrak{g}$ and the only non-trivial component of the differential the identity map of $\mathfrak{g}$.

\end{example}

\section{Proofs}\label{section: proofs}
\subsection{Lemma \ref{lemma: LA structure on pr-ev}}\label{subsection: proof of LA structure descends}
Suppose that
$f\in\mathcal{O}_X,\,\omega_1,\omega_2\in\mathcal{O}_{X^{\sharp}},\,a_1,a_2\in\mathcal{A}$.
The identities
\begin{multline}\label{lemma: LA structure on pr-ev 1}
[\omega_1\otimes1\otimes fa_1,\omega_2\otimes1\otimes a_2]= \\
=[f\omega_1\otimes1\otimes a_1,\omega_2\otimes1\otimes
a_2]+[\omega_1\wedge df\otimes\epsilon\otimes
a_1,\omega_2\otimes1\otimes a_2];
\end{multline}
\begin{multline}\label{lemma: LA structure on pr-ev 2}
[\omega_1\otimes1\otimes fa_1,\omega_2\otimes\epsilon\otimes a_2]= \\
=[f\omega_1\otimes1\otimes a_1,\omega_2\otimes\epsilon\otimes
a_2]+[\omega_1\wedge df\otimes\epsilon\otimes
a_1,\omega_2\otimes\epsilon\otimes a_2];
\end{multline}
\begin{equation}\label{lemma: LA structure on pr-ev 3}
[f\omega_1\otimes\epsilon\otimes a_1,
\omega_2\otimes\epsilon\otimes
a_2]=[\omega_1\otimes\epsilon\otimes fa_1,
\omega_2\otimes\epsilon\otimes a_2];
\end{equation}
\begin{equation}\label{lemma: LA structure on pr-ev 4}
[f\omega_1\otimes\epsilon\otimes a_1, \omega_2\otimes1\otimes
a_2]=[\omega_1\otimes\epsilon\otimes fa_1, \omega_2\otimes1\otimes
a_2];
\end{equation}
show that the bracket on
$\mathcal{O}_{X^{\sharp}}[\epsilon]\otimes_{\mathbb{C}}\mathcal{A}$
descends to $\pr_*\ev^*\mathcal{A}$. We will verify equations
\eqref{lemma: LA structure on pr-ev 1} and \eqref{lemma: LA
structure on pr-ev 4} leaving \eqref{lemma: LA structure on pr-ev
2} and \eqref{lemma: LA structure on pr-ev 3} to the reader.
\begin{align*}
[\omega_1\otimes1&\otimes fa_1,\omega_2\otimes1\otimes a_2]=\omega_1\wedge(fL_{\sigma(a_1)}\omega_2+df\wedge(\iota_{\sigma(a_1)}\omega_2)\otimes1\otimes a_2+\\
&+(-1)^{\omega_1\omega_2}\omega_2\wedge\omega_1\otimes1\otimes
f[a_1,a_2]-(-1)^{\omega_1\omega_2}\omega_2\wedge\omega_1\otimes1\otimes
L_{\sigma(a_2)}f\cdot a_1+\\
& -(-1)^{\omega_1\omega_2}\omega_2\wedge
L_{\sigma(a_2)}\omega_1\otimes 1 \otimes f\cdot a_1 = \\
& = [f\omega_1\otimes1\otimes a_1,\omega_2\otimes1\otimes
a_2]+[\omega_1\wedge df\otimes\epsilon\otimes
a_1,\omega_2\otimes1\otimes a_2]
\end{align*}
\begin{align*}
[f\omega_1&\otimes\epsilon\otimes a_1,
\omega_2\otimes1\otimes a_2] = f\omega_1\wedge(\iota_{\sigma(a_1)}\omega_2)\otimes1\otimes a_2+ \\
& +(-1)^{(\omega_1-1)\omega_2}\omega_2\wedge
f\omega_1\otimes\epsilon\otimes
[a_1,a_2]-(-1)^{(\omega_1-1)\omega_2}\omega_2\wedge
L_{\sigma(a_2)}(f\omega_1)\otimes\epsilon\otimes a_1=\\
& = [\omega_1\otimes\epsilon\otimes fa_1, \omega_2\otimes1\otimes
a_2].
\end{align*}
The calculations
\[
\sigma(1\otimes\epsilon\otimes
fa)=\iota_{\sigma(fa)}=f\iota_{\sigma(a)}=\sigma(f\otimes\epsilon\otimes
a),
\]
\begin{align*}
\sigma(1\otimes1\otimes fa) & =
L_{\sigma(fa)}=fL_{\sigma(a)}+df\wedge\iota_{\sigma(a)}\\
& = \sigma(f\otimes1\otimes a)+\sigma(df\otimes\epsilon\otimes a)
\end{align*}
show that the anchor map descends to $\pr_*\ev^*\mathcal{A}$.

\subsection{Lemma \ref{lemma: iota descends}}\label{subsection: proof of iota descends}
{~}\\
(1) Suppose that $f\in\mathcal{O}_{X}$, $\alpha\in\Omega^{k+1}_X$
and $B\in(\pr_*\ev^{*}\mathcal{E})^{0}$. From the computations
\begin{align*}
\widetilde{\iota_{D}}(f\alpha\otimes B) & =
\iota_{\sigma(D)}(f\alpha)\otimes
B+(-1)^{\alpha}f\alpha\otimes\iota_{D}B\\
& = f\iota_{D}(\alpha)\otimes
B+(-1)^{\alpha}f\alpha\otimes\iota_{D}B\\
& = f(\widetilde{\iota_{D}}\alpha\otimes B)
\end{align*}
and
\begin{align*}
\widetilde{\iota_{D}}(\alpha\otimes fB)& =
\iota_{\sigma(D)}(\alpha)\otimes
fB+(-1)^{\alpha}\alpha\otimes\iota_{D}fB\\
& = f\widetilde{\iota_{D}}(\alpha\otimes B),
\end{align*} it follows that $\widetilde{\iota_{D}}(\alpha\otimes fB) = \widetilde{\iota_{D}}(f\alpha\otimes
B)$. \\

\noindent (2)  Suppose that $\gamma\in\Omega^{k}_X$ and
$B\in(\pr_*\ev^*\mathcal{E})$. The principal symbol of $\widetilde{\iota_D}$ is equal to
\begin{align*}
[\widetilde{\iota_{D}},\gamma](B) & =
(\widetilde{\iota_{D}}\circ\gamma-(-1)^{\gamma}\gamma\circ\widetilde{\iota_{D}})B\\
& =
\iota_{\sigma(D)}\gamma\otimes(B)+(-1)^{\gamma}\gamma\widetilde{\iota_{D}}(B)-(-1)^{\gamma}\gamma\widetilde{\iota_{D}}(B)\\
& = \iota_{\sigma(D)}\gamma\otimes(B)\\
& = (\sigma(D)\otimes\epsilon)\gamma\otimes B.
\end{align*}

\subsection{Lemma \ref{lemma: Atiyah action}}\label{subsection: proof of Atiyah action}
{~}\\
(1)
\begin{align*}
\widetilde{D}(A\otimes e) & =
[\partial,\widetilde{\iota_{D}}](A\otimes e)\\
& =
(\partial\circ\widetilde{\iota_{D}}+\widetilde{\iota_{D}}\circ\partial)(A\otimes
e)\\
& = d(\iota_{\sigma(D)}A\otimes e+(-1)^A A\epsilon\otimes
De)+\widetilde{\iota_{D}}(\partial A\otimes e)\\
& = d\iota_{\sigma(D)}A\otimes e+A\otimes
De+\iota_{\sigma(D)}d A\otimes e\\
& = L_{\sigma(D)}A\otimes e + A\otimes De
\end{align*}

\noindent (2) Suppose that $D_1,D_2\in\mathcal{A}_{\mathcal{E}}$. The calculation
\begin{align*}
[\widetilde{D_1},\widetilde{D_2}](A\otimes e) & = (\widetilde{D_1}\circ\widetilde{D_2}-\widetilde{D_2}\circ\widetilde{D_1})(A\otimes e) \\
& = \widetilde{D_1}(L_{\sigma(D_2)}A\otimes e+ A\otimes D_2e)-\widetilde{D_2}(L_{\sigma(D_1)}A\otimes e+ A\otimes D_1e)\\
& = L_{\sigma(D_1)}L_{\sigma(D_2)}A\otimes e + L_{\sigma(D_2)}A\otimes D_1e \\
& + L_{\sigma(D_1)}A\otimes D_2e+ A\otimes D_1D_2e - L_{\sigma(D_2)}L_{\sigma(D_1)}A\otimes e\\
& - L_{\sigma(D_1)}A\otimes D_2e- L_{\sigma(D_2)}A\otimes D_1e- A\otimes D_2D_1e\\
& = [L_{\sigma(D_1)},L_{\sigma(D_2)}]A\otimes e+ A\otimes[D_1,D_2]e\\
& = L_{[\sigma(D_1),\sigma(D_2)]}A\otimes e+ A\otimes[D_1,D_2]e\\
& = \widetilde{[D_1,D_2]}(A\otimes e)
\end{align*}
shows that the map $D\mapsto\widetilde{D}$ is compatible with
brackets.\\

\noindent (3) The commutativity of the diagram follows from the first item.\\

\noindent (4) Left to the reader.

\subsection{Lemma \ref{lemma: morphism of LA}}\label{subsection: proof of morphism of LA}
{~}\\
(1) Suppose that $D\in\mathcal{A}_{\mathcal{E}}$ and
$\omega_1,\omega_2\in\mathcal{O}_{X^{\sharp}}$. By definition of
$L^{\mathcal{E}}$
\[
L^{\mathcal{E}}((\omega_1+\omega_2\otimes\epsilon)\otimes
D)=\omega_1\otimes\widetilde{D}+\omega_2\otimes\widetilde{\iota_{D}},
\]
For each $f\in\mathcal{O}_X$ the formula $\widetilde{fD}=df\otimes\widetilde{\iota_{D}}+f\widetilde{D}$ holds.
The calculation
\begin{align*}
L^{\mathcal{E}}((\omega_1+\omega_2\otimes\epsilon)\otimes fD) & =
\omega_1\otimes\widetilde{fD}+\omega_2\otimes f\widetilde{\iota_D}\\
& = f\omega_1\otimes\widetilde{D}+(\omega_1+ f\omega_2)\otimes\widetilde{\iota_{D}}\\
& = L^{\mathcal{E}}((f\omega_1+(f\omega_2+\omega_1\wedge df)\epsilon)\otimes D)\\
& = L^{\mathcal{E}}((\omega_1+\omega_2\otimes\epsilon)(f+df\otimes\epsilon)\otimes D),
\end{align*}
shows that the map $L^{\mathcal{E}}$ descends to the map
$\mathcal{A}^{\sharp}_{\mathcal{E}}\to\mathcal{A}_{\pr_*\ev^*\mathcal{E}}$.\\

\noindent (2) Let
$\omega_1,\omega_2,\gamma_1,\gamma_2\in\mathcal{O}_{X^{\sharp}}$
and $D_1,D_2\in\mathcal{A}_{\mathcal{E}}$. We leave verification
of the identities
$[\widetilde{D_1},\widetilde{\iota_{D_2}}]=\widetilde{\iota_{[D_1,D_2]}}$
and $[\widetilde{\iota_{D_1}},\widetilde{\iota_{D_2}}]=0$ to the
reader. The calculation
\begin{align*}
L^{\mathcal{E}} & ([\omega_1\otimes D_1+ \gamma_1\otimes\epsilon\otimes D_1, \omega_2\otimes D_2+\gamma_2\otimes\epsilon\otimes D_2])=\\
& =L^{\mathcal{E}}(\omega_1\wedge L_{\sigma(D_1)}\omega_2\otimes D_2+(-1)^{\omega_1\omega_2}\omega_2\wedge\omega_1\otimes[D_1,D_2]+ \\
& -(-1)^{\omega_1\omega_2}\omega_2\wedge L_{\sigma(D_2)}\omega_1\otimes D_1)+L^{\mathcal{E}}(\omega_1\wedge L_{\sigma(D_1)}\gamma_2\otimes\epsilon\otimes D_2+\\
& +(-1)^{\omega_1\omega_2}\gamma_2\wedge\omega_1\otimes\epsilon\otimes[D_1,D_2]-(-1)^{\omega_1\omega_2+\omega_1}\gamma_2\wedge\iota_{\sigma(D_2)}\omega_1\otimes D_1)+\\
& +L^{\mathcal{E}}(\gamma_1\wedge\iota_{\sigma(D_1)}\omega_2\otimes D_2-(-1)^{\omega_2\omega_1}\omega_2\wedge\gamma_1\otimes\epsilon\otimes[D_1,D_2]+\\
& +(-1)^{\omega_1\omega_2+\omega_1}\omega_2\wedge L_{\sigma(D_2)}\gamma_1\otimes\epsilon\otimes D_1)+L^{\mathcal{E}}(\gamma_1\wedge\iota_{\sigma(D_1)}\gamma_2\otimes\epsilon\otimes D_2+\\
&+(-1)^{\omega_1\omega_2+\omega_2+\omega_1}\gamma_2\wedge\iota_{\sigma(D_2)}\gamma_1\otimes\epsilon\otimes D_1)=\\
& = \omega_1\wedge L_{\sigma(D_1)}\omega_2\otimes\widetilde{D_2} +(-1)^{\omega_1\omega_2}\omega_2\wedge\omega_1\otimes\widetilde{[D_1,D_2]}+\\
& -(-1)^{\omega_1\omega_2}\omega_2\wedge L_{\sigma(D_2)}\omega_1\otimes\widetilde{D_1} +\omega_1\wedge L_{\sigma(D_1)}\gamma_2\otimes\widetilde{\iota_{D_2}}+\\
& +(-1)^{\omega_1\omega_2}\gamma_2\wedge\omega_1\otimes\widetilde{\iota_{[D_1,D_2]}}-(-1)^{\omega_1\omega_2+\omega_1}\gamma_2\wedge\iota_{\sigma(D_2)}\omega_1\otimes\widetilde{D_1}+\\
& +\gamma_1\wedge\iota_{\sigma(D_1)}\omega_2\otimes\widetilde{D_2}+(-1)^{\omega_2\omega_1}\omega_2\wedge\gamma_1\otimes\widetilde{\iota_{[D_1,D_2]}}+\\
& -(-1)^{\omega_1\omega_2+\omega_1}\omega_2\wedge L_{\sigma(D_2)}\gamma_1\otimes\widetilde{\iota_{D_1}}+ \gamma_1\wedge\iota_{\sigma(D_1)}\gamma_2\otimes\widetilde{\iota_{D_2}}+\\
&+(-1)^{\omega_1\omega_2+\omega_2+\omega_1}\gamma_2\wedge\iota_{\sigma(D_2)}\gamma_1\otimes\widetilde{\iota_{D_1}}= \\
& =[L^{\mathcal{E}}(\omega_1\otimes D_1+ \gamma_1\otimes\epsilon\otimes D_1), L^{\mathcal{E}}(\omega_2\otimes D_2+\gamma_2\otimes\epsilon\otimes D_2)]
\end{align*}
shows that the map $L^{\mathcal{E}}$ is a map of Lie algebras. We
leave to the reader to check that \eqref{Lie-E} commutes with
respective anchor maps. Hence it is a map of
$\mathcal{O}_{X^{\sharp}}$-Lie algebroids.
\subsection{Lemma \ref{lemma: bracket Leibniz}}\label{subsection: proof bracket Leibniz}
Suppose that $a=\omega_1+\alpha_1\epsilon\otimes
q_1+\beta_1\otimes r_1$ and $b=\omega_2+\alpha_2\epsilon\otimes
q_2+\beta_2\otimes r_2$ are homogeneous elements, where
$\omega_i,\alpha_i,\beta_i\in\mathcal{O}_{X^{\sharp}}$ and
$q_i,r_i\in\mathcal{Q}$ for $i=1,2$. The calculation
\begin{align*}
\partial & ([\alpha_1\epsilon\otimes q_2,\alpha_2\epsilon\otimes q_2]) = \partial(\alpha_1\wedge(\iota_{\pi(q_1)}\alpha_2)\epsilon\otimes q_2) \\
&
+(-1)^{\alpha_1\alpha_2+\alpha_1+\alpha_2}\partial(\alpha_2\wedge(\iota_{\pi(q_2)}\alpha_1)\epsilon\otimes
q_1)+(-1)^{\alpha_1\alpha_2+\alpha_2}d(\alpha_2\wedge\alpha_1\wedge\langle
q_1,q_2\rangle)\\
&=d(\alpha_1\wedge(\iota_{\pi(q_1)}\alpha_2))\epsilon\otimes
q_2+(-1)^{\alpha_1+\alpha_2-1}\alpha_2\wedge(\iota_{\pi(q_2)})\alpha_1\otimes
q_2\\
&
+(-1)^{(\alpha_1-1)\alpha_2}d(\alpha_2\wedge\alpha_1\wedge\langle
q_1,q_2\rangle)+(-1)^{\alpha_1\alpha_2+\alpha_1-\alpha_2}(d(\alpha_2\wedge(\iota_{\pi_{q_2}}\alpha_1))\epsilon\otimes
q_1 \\
&
+(-1)^{\alpha_2+\alpha_1-1}\alpha_2\wedge(\iota_{\pi(q_2)}\alpha_1)\otimes
q_1) \\
& = [\partial(\alpha_1\epsilon\otimes q_1),\alpha_2\epsilon\otimes
q_2]+(-1)^{(\alpha_1-1)}[\alpha_1\epsilon\otimes q_1,
\partial(\alpha_2\epsilon\otimes q_2)],
\end{align*}
shows that \eqref{bracket and differential} holds in the case
\eqref{bracket with scalars-1} of the operation \eqref{bracket no
scalars}. We leave the cases \eqref{bracket with
scalars-2}-\eqref{bracket with scalars-4} to the reader.

\subsection{Lemma \ref{lemma: two-sided ideal}} \label{subsection: two-sided ideal}
Suppose that $a=\omega+\alpha\epsilon\otimes q+\beta\otimes
r\in\mathcal{O}_{X^\sharp}[k]\oplus
\mathcal{O}_{X^\sharp}[\epsilon]\otimes_\mathbb{C}\mathcal{Q}$ is
a homogeneous element and $b=\gamma\otimes fs-f\gamma\otimes
s-\gamma\wedge df \epsilon \otimes s\in K$, where
$\omega,\alpha,\beta,\gamma\in\mathcal{O}_{X^{\sharp}}$,
$f\in\mathcal{O}_X$ and $q,r,s\in\mathcal{Q}$. Then,
$$[a,b]=[\omega,b]+[\alpha\epsilon\otimes q,b]+[\beta\otimes r,b].$$ The formulas \eqref{bracket with
scalars-5} and \eqref{bracket with scalars-6} imply that
$[\omega,b]\in K$. The calculations,
\begin{multline*}
[\alpha\epsilon\otimes q,b]=[\alpha\epsilon\otimes q,
\gamma\otimes fs-f\gamma\otimes s-\gamma\wedge df
\epsilon \otimes s] \\
=\alpha\wedge\iota_{\pi(q)}\gamma(1\otimes fs-f\otimes s-df
\epsilon \otimes
s)+(-1)^{(\alpha-1)\gamma}\gamma\wedge\alpha(1\epsilon
f\{q,s\} \\
-f \epsilon \{q,s\})+(-1)^{\gamma}\alpha\wedge\gamma(1 \epsilon
\otimes (L_{\pi(q)})fs-\iota_{\pi(q)}(df) \epsilon \otimes s)
\end{multline*}
and
\begin{multline*}
[\beta\otimes r, b]=[\beta\otimes r, \gamma\otimes
fs-f\gamma\otimes s-\gamma\wedge df
\epsilon \otimes s] \\
=\beta\wedge (L_{\pi(r)}\gamma)(1\otimes fs-f\otimes
s-df\epsilon\otimes
s)+(-1)^{\beta\gamma}\gamma\wedge\beta(1\otimes(L_{\pi(r)}f)s \\
-(L_{\pi(r)}f)\otimes s-d(L_{\pi(r)}f) \epsilon \otimes
s)+(-1)^{\beta\gamma}\gamma\wedge\beta(1\otimes
f\{r,s\}-f\otimes \{r,s\}\\
-df \epsilon \otimes \{r,s\})
\end{multline*}
show that $[\alpha\epsilon\otimes q,b],[\beta\otimes r, b]\in K$.
We leave it to the reader to verify that $[b,a]\in K$.

\subsection{Proposition \ref{tau as a marked Lie}}\label{subsection: tau as a marked
Lie} The proof of the proposition proceeds in several steps. All
elements are assumed to be homogeneous. Suppose that
$\omega,\alpha_i \in\mathcal{O}_{X^{\sharp}}$,
$q,q_i\in\mathcal{Q}$ for $i=1,2$.
\subsubsection{Skew-symmetry}
Calculations
\[
[\alpha_1\epsilon\otimes q_1,\alpha_2\epsilon\otimes
q_2]+(-1)^{(\alpha_1-1)(\alpha_2-1)}[\alpha_2\epsilon\otimes
q_2,\alpha_1\epsilon\otimes q_1]=0,
\]
\begin{multline*}
[\alpha_1\otimes q_1,\alpha_2\epsilon\otimes
q_2]+(-1)^{\alpha_1(\alpha_2-1)}[\alpha_2\epsilon\otimes
q_2,\alpha_1\otimes
q_1]=\\
(-1)^{\alpha_1\alpha_2}\alpha_2\wedge\alpha_1(\epsilon\otimes\pi^{\dagger}(d\langle
q_1,q_2\rangle)-d\langle q_1,q_2\rangle)=0,
\end{multline*}
\begin{multline*}
[\alpha_1\otimes q_1,\alpha_2\otimes
q_2]+(-1)^{\alpha_1\alpha_2}[\alpha_2\otimes q_2,\alpha_1\otimes
q_1]= \\
(-1)^{\alpha_1\alpha_2}\alpha_2\wedge\alpha_1\otimes
\pi^{\dagger}(d\langle q_1,q_2\rangle)=0,
\end{multline*}
\[
[\alpha\epsilon\otimes
q,\omega]+(-1)^{(\alpha-1)\omega}[\omega,\alpha\epsilon\otimes
q]=0,
\]
\[
[\alpha\otimes q,\omega]+(-1)^{\alpha\omega}[\omega,\alpha\otimes
q]=0
\]
show that symmetrization of the brackets given by formulas
\eqref{bracket with scalars-1}--\eqref{bracket with scalars-6} is
trivial. Thus the operation \eqref{bracket no scalars} is
skew-symmetric on $\tau\mathcal{Q}$.

\subsubsection{Jacobi identity} We outline the proof of the Jacobi identity leaving details to the reader. It is easy to
see that Jacobi identity holds for the rules $(1)$-$(6)$ that
define the operation \eqref{bracket no scalars} on
$\tau\mathcal{Q}$. By the linearity of the operation
\eqref{bracket no scalars}, it is enough to check on it the Jacobi
identity for the combination of elements:
\begin{multline}\label{jacobi identity 1}
[\alpha_1\epsilon\otimes q_1,[\alpha_2\epsilon\otimes
q_2,\alpha_3\epsilon\otimes q_3]]=\\
[[\alpha_1\epsilon\otimes q_1,\alpha_2\epsilon\otimes
q_2],\alpha_3\epsilon\otimes q_3]\pm
[\alpha_2\epsilon\otimes q_2,[\alpha_1\epsilon\otimes
q_1,\alpha_3\epsilon\otimes q_3]]
\end{multline}
\begin{multline}\label{jacobi identity 2}
[\alpha_1\otimes q_1,[\alpha_2\otimes
q_2,\alpha_3\otimes q_3]]=\\
[[\alpha_1\otimes
q_1,\alpha_2\otimes q_2],\alpha_3\otimes
q_3]\pm[\alpha_2\otimes q_2,[\alpha_1\otimes
q_1,\alpha_3\otimes q_3]]
\end{multline}
\begin{multline}\label{jacobi identity 3}
[\alpha_1\epsilon\otimes q_1,[\alpha_2\otimes
q_2,\alpha_3\otimes q_3]]= \\
[[\alpha_1\epsilon\otimes
q_1,\alpha_2\otimes q_2],\alpha_3\otimes
q_3]\pm[\alpha_2\otimes q_2,[\alpha_1\epsilon\otimes
q_1,\alpha_3\otimes q_3]]
\end{multline}
\begin{multline}\label{jacobi identity 4}
[\alpha_1\epsilon\otimes q_1,[\alpha_2\epsilon\otimes q_2,\alpha_3\otimes q_3]]= \\
[[\alpha_1\epsilon\otimes q_1,\alpha_2\epsilon\otimes
q_2],\alpha_3\otimes q_3]\pm[\alpha_2\epsilon\otimes
q_2,[\alpha_1\epsilon\otimes q_1,\alpha_3\otimes q_3]]
\end{multline}
\begin{multline}\label{jacobi identity 5}
[\alpha_1 \epsilon \otimes q_1,[\alpha_2\otimes
q_2,\alpha_3 \epsilon \otimes q_3]]= \\
[[\alpha_1 \epsilon \otimes q_1,\alpha_2\otimes
q_2],\alpha_3\epsilon \otimes q_3] \pm [\alpha_2
\otimes q_2,[\alpha_1\epsilon\otimes q_1,\alpha_3
\epsilon \otimes q_3]] \\
\end{multline}
\begin{multline}\label{jacobi identity 6}
[\alpha_1 \otimes q_1,[\alpha_2\otimes q_2,\alpha_3\epsilon\otimes q_3]]=\\
[[\alpha_1\otimes q_1,\alpha_2\otimes q_2],\alpha_3\epsilon\otimes q_3]\pm[\alpha_2\otimes q_2,[\alpha_1\otimes q_1,\alpha_3\epsilon\otimes q_3]]
\end{multline}
\begin{multline}\label{jacobi identity 7}
[\alpha_1 \otimes q_1,[\alpha_2\epsilon\otimes q_2,\alpha_3\otimes q_3]]=\\
[[\alpha_1\otimes q_1,\alpha_2\epsilon\otimes q_2],\alpha_3\otimes q_3]\pm[\alpha_2\epsilon\otimes q_2,[\alpha_1\otimes q_1,\alpha_3\otimes q_3]]
\end{multline}
\begin{multline}\label{jacobi identity 8} [\alpha_1\otimes q_1,[\alpha_2\epsilon\otimes
q_2,\alpha_3 \epsilon \otimes q_3]]= \\
[[\alpha_1\otimes q_1,\alpha_2\epsilon\otimes
q_2],\alpha_3\epsilon \otimes q_3] \pm [\alpha_1
\otimes q_1,[\alpha_2\epsilon\otimes q_2,\alpha_3
\epsilon \otimes q_3]] \\
\end{multline}
Since \eqref{ip invariance} holds in $\mathcal{Q}$, \eqref{jacobi
identity 1} and \eqref{jacobi identity 2} are deduced using
formulas \eqref{bracket with scalars-1}-\eqref{bracket with
scalars-6} and the Leibniz rule. The verification of \eqref{jacobi identity
3}-\eqref{jacobi identity 7} is reported in the next
calculations.\\
The left hand side of equation \eqref{jacobi identity 3} is equal
to:
\begin{align*}
[\alpha_1 & \epsilon\otimes q_1,[\alpha_2\otimes
q_2,\alpha_3\otimes q_3]]= \alpha_1\wedge\iota_{\pi(q_1)}(\alpha_2 L_{\pi(q_2)}\alpha_3)\otimes q_3\pm \\
& \alpha_2\wedge(L_{\pi(q_2)}\alpha_3)\wedge\alpha_1 d\langle
q_1,q_3\rangle\pm
\alpha_2\wedge(L_{\pi(q_2)}\alpha_3)\wedge\alpha_1\epsilon\otimes\{q_1,q_3\}\pm\\
&
\alpha_2\wedge(L_{\pi(q_2)}\alpha_3)\wedge(L_{\pi(q_3)}\alpha_1)\epsilon\otimes
q_1
\alpha_1\wedge\iota_{\pi(q_1)}(\alpha_3\alpha_2)\{q_2,q_3\}\pm\alpha_3\wedge\alpha_2\wedge\alpha_1\wedge
d\langle q_1,\{q_2,q_3\}\rangle\pm\\
& \alpha_3\wedge\alpha_2\wedge\alpha_1\epsilon\otimes\{q_1,\{q_2,q_3\}\}
\pm
\alpha_3\wedge\alpha_2\wedge(L_{[\pi(q_2),\pi(q_3)]}\alpha_1)\epsilon\otimes
q_1\pm\\
&
\alpha_1\wedge\iota_{\pi(q_1)}(\alpha_3\wedge L_{\pi(q_3)}\alpha_2)\otimes
q_2\pm \alpha_3\wedge(L_{\pi(q_3)}\alpha_2)\wedge\alpha_1\wedge d\langle q_1,q_2
\rangle\pm\\
&
\alpha_3\wedge(L_{\pi(q_3)}\alpha_2)\wedge\alpha_1\epsilon\otimes\{q_1,q_2\}\pm
\alpha_3\wedge(L_{\pi(q_3)}\alpha_2)\wedge(L_{\pi(q_2)}\alpha_1)\epsilon\otimes
q_1.
\end{align*}
In the right hand side of \eqref{jacobi identity 3} the first
summand is equal to,
\begin{align*}
[[\alpha_1 & \epsilon\otimes q_1,\alpha_2\otimes
q_2],\alpha_3\otimes q_3]=
\alpha_1\wedge(\iota_{\pi(q_1)}\alpha_2)\wedge(L_{\pi(q_2)}\alpha_3)\otimes q_3\pm \\
&
\alpha_3\wedge\alpha_1\wedge(\iota_{\pi(q_1)}\alpha_2)\otimes\{q_2,q_3\}\pm
\alpha_3\wedge L_{\pi(q_3)}(\alpha_1\wedge\iota_{\pi(q_1)}\alpha_2)\otimes
q_2\pm\\
& \alpha_3\wedge L_{\pi(q_3)}(\alpha_2\wedge\alpha_1\wedge d\langle
q_1,q_2\rangle)\pm
\alpha_2\wedge\alpha_1\wedge(\iota_{\{q_1,q_2\}}\alpha_3)\otimes q_3\pm\\
& \alpha_3\wedge\alpha_2\wedge\alpha_1\wedge d\langle\{q_1,q_2\},q_3\rangle\pm
\alpha_3\wedge\alpha_2\wedge\alpha_1\epsilon\otimes\{\{q_1,q_2\},q_3\}\pm\\
&
\alpha_3\wedge L_{\pi(q_3)}(\alpha_2\wedge\alpha_1)\epsilon\otimes\{q_1,q_2\}\pm
\alpha_2\wedge(L_{\pi(q_2)}\alpha_1)\wedge(\iota_{\pi(q_1)}\alpha_3)\otimes
q_3\pm\\
& \alpha_3\wedge\alpha_2\wedge(L_{\pi(q_2)}\alpha_2)\wedge d\langle q_1,q_3\rangle\pm
\alpha_3\wedge\alpha_2\wedge(L_{\pi(q_2)}\alpha_1)\epsilon\otimes
\{q_1,q_3\} \pm\\
& \alpha_3\wedge L_{\pi(q_3)}(\alpha_2\wedge
L_{\pi(q_2)}\alpha_1)\epsilon\otimes q_1.
\end{align*}
The second summand is equal to:
\begin{align*}
 [\alpha_2 & \otimes q_2,[\alpha_1\epsilon\otimes
q_1,\alpha_3\otimes q_3]]=
\alpha_2\wedge L_{\pi(q_2)}(\alpha_1\wedge\iota_{\pi(q_1)}\alpha_3)\otimes
q_3\pm\\
&
\alpha_1\wedge(\iota_{\pi(q_1)}\alpha_3)\wedge\alpha_2\otimes\{q_2,q_3\}\pm
\alpha_1\wedge(\iota_{q_1}\alpha_3)\wedge(L_{\pi(q_3)}\alpha_2)\otimes
q_2\pm \\
& \alpha_{2}\wedge L_{\pi(q_2)}(\alpha_3\wedge\alpha_1\wedge d\langle q_1,q_3\rangle)
\pm
\alpha_2\wedge L_{\pi(q_2)}(\alpha_3\wedge\alpha_1)\epsilon\otimes\{q_1,q_3\}\pm
\\
&
\alpha_3\wedge\alpha_1\wedge\alpha_2\epsilon\otimes\{q_2,\{q_1,q_3\}\}\pm
\alpha_3\wedge\alpha_1\wedge(\iota_{[\pi(q_1),\pi(q_3)]}\alpha_2)\otimes
q_2\pm\\
& \alpha_2\wedge L_{\pi(q_2)}(\alpha_3\wedge
L_{\pi(q_3)}\alpha_1)\epsilon\otimes q_1\pm
\alpha_3\wedge(L_{\pi(q_3)}\alpha_1)\wedge\alpha_2\epsilon\otimes\{q_2,q_1\}\pm\\
&
\alpha_3\wedge(L_{\pi(q_3)}\alpha_1)\wedge(\iota_{\pi(q_1)}\alpha_2)\otimes
q_2.
\end{align*}
Using Leibniz rule the above calculations reduce \eqref{jacobi identity 3} to
\begin{multline*}
\alpha_3\wedge(L_{\pi(q_3)}\alpha_2)\wedge\alpha_1\wedge d\langle q_1,q_2
\rangle\pm
\alpha_3\wedge(L_{\pi(q_3)}\alpha_2)\wedge\alpha_1\epsilon\otimes\{q_1,q_2\}=\\
\alpha_3\wedge L_{\pi(q_3)}(\alpha_2\wedge\alpha_1)\epsilon\otimes\{q_1,q_2\}\pm\alpha_3\wedge(L_{\pi(q_3)}\alpha_1)\wedge\alpha_2\epsilon\otimes\{q_2,q_1\}
\end{multline*}
The latter equation follows from $\epsilon\otimes\{q_2,q_1\}=\epsilon\otimes(-\{q_1,q_2\}+\pi^{\dagger}d\langle
q_1,q_2\rangle)$.\\
In the case of equation \eqref{jacobi identity 4}, the left hand
side is equal to,
\begin{align*}
& [\alpha_1\epsilon\otimes q_1,[\alpha_2\epsilon\otimes q_2,\alpha_3\otimes q_3]]= \\
&
\alpha_1\wedge\iota_{\pi(q_1)}(\alpha_2\wedge\iota_{\pi(q_2)}\alpha_3)\otimes
q_3\pm \alpha_2\wedge(\iota_{\pi(q_2)}\alpha_3)\wedge\alpha_1\wedge d\langle
q_1,q_3\rangle \pm \\
&
\alpha_2\wedge(\iota_{\pi(q_2)}\alpha_3)\wedge\alpha_1\epsilon\otimes\{q_1,q_3\}\pm
\alpha_2\wedge(\iota_{\pi(q_2)}\alpha_3)\wedge(L_{\pi(q_3)}\alpha_1)\epsilon\otimes
q_1\pm \\
& \alpha_1\wedge\iota_{\pi(q_1)}(\alpha_3\wedge\alpha_2\wedge d\langle
q_2,q_3\rangle) \pm
\alpha_1\wedge\iota_{\pi(q_1)}(\alpha_3\wedge\alpha_2)\epsilon\otimes\{q_2,q_3\}\pm
\\
& \alpha_3\wedge\alpha_2\wedge\alpha_1\wedge\langle q_1,\{q_2,q_3\}\rangle\pm
\alpha_3\wedge\alpha_2\wedge\iota_{[\pi(q_2),\pi(q_3)]}\alpha_3\epsilon\otimes
q_1\pm \\
&
\alpha_1\wedge\iota_{\pi(q_1)}(\alpha_3\wedge L_{\pi(q_3)}\alpha_2)\epsilon\otimes
q_2\pm \alpha_3\wedge(L_{\pi(q_3)}\alpha_2)\wedge\alpha_1\langle
q_1,q_2\rangle \pm \\
&
\alpha_3\wedge(L_{\pi(q_3)}\alpha_2)\wedge(\iota_{\pi(q_2)}\alpha_1)\epsilon\otimes
q_1
\end{align*}
In the right hand side the first summand is equal to
\begin{align*}
& [[\alpha_1\epsilon\otimes q_1,\alpha_2\epsilon\otimes
q_2],\alpha_3\otimes q_3]= \\
&
\alpha_1\wedge(\iota_{\pi(q_1)}\alpha_2)\wedge(\iota_{\pi(q_2)}\alpha_3)\otimes
q_3\pm \alpha_3\wedge\alpha_1\wedge(\iota_{\pi(q_1)}\alpha_2)\wedge d\langle
q_2,q_3\rangle \pm \\
&
\alpha_3\wedge\alpha_1\wedge(\iota_{\pi(q_1)}\alpha_2)\epsilon\otimes\{q_2,q_3\}\pm
\alpha_3\wedge L_{\pi(q_3)}(\alpha_1\wedge\iota_{\pi(q_1)}\alpha_2)\epsilon\otimes
q_2\pm \\
& \alpha_3\wedge L_{\pi(q_3)}(\alpha_2\wedge\alpha_1\wedge\langle q_1,q_2\rangle)\pm
\alpha_2\wedge(\iota_{\pi(q_2)}\alpha_1)\wedge(\iota_{\pi(q_1)}\alpha_3)\otimes
q_3\pm\\
& \alpha_3\wedge\alpha_2\wedge(\iota_{\pi(q_2)}\alpha_1)\wedge d\langle
q_1,q_3\rangle \pm
\alpha_3\wedge\alpha_2\wedge(\iota_{\pi(q_2)}\alpha_1)\epsilon\otimes\{q_1,q_3\}\pm\\
&
\alpha_3\wedge L_{\pi(q_3)}(\alpha_2\wedge(\iota_{\pi(q_2)}))\epsilon\otimes
q_1
\end{align*}
The second summand is equal to,
\begin{align*}
& [\alpha_2\epsilon\otimes q_2,[\alpha_1\epsilon\otimes
q_1,\alpha_3\otimes q_3]]= \\
&
\alpha_2\wedge\iota_{\pi(q_2)}(\alpha_1\wedge\iota_{\pi(q_1)}\alpha_3)\otimes
q_3\pm \alpha_1\wedge(\iota_{\pi(q_1)}\alpha_3)\wedge\alpha_2\wedge d\langle
q_2,q_3\rangle \pm \\
&
\alpha_1\wedge(\iota_{\pi(q_1)}\alpha_3)\wedge\alpha_2\epsilon\otimes\{q_2,q_3\}\pm
\alpha_1\wedge(\iota_{\pi(q_1)}\alpha_3)\wedge(L_{\pi(q_3)}\alpha_2)\epsilon\otimes
q_2\pm \\
& \alpha_2\wedge\iota_{\pi(q_2)}(\alpha_3\wedge\alpha_1\wedge d\langle
q_1,q_3\rangle) \pm
\alpha_2\wedge\iota_{\pi(q_2)}(\alpha_3\wedge\alpha_1)\epsilon\otimes\{q_1,q_3\}\pm
\\
& \alpha_3\wedge\alpha_1\wedge\alpha_2\wedge\langle q_2,\{q_1,q_3\}\rangle\pm
\alpha_3\wedge\alpha_1\wedge\iota_{[\pi(q_1),\pi(q_3)]}\alpha_2\epsilon\otimes
q_2\pm \\
&
\alpha_2\wedge\iota_{\pi(q_2)}(\alpha_3\wedge L_{\pi(q_3)}\wedge\alpha_1)\epsilon\otimes
q_1\pm \alpha_3\wedge(L_{\pi(q_3)}\alpha_1)\wedge\alpha_2\wedge\langle
q_1,q_2\rangle \pm \\
&
\alpha_3\wedge(L_{\pi(q_3)}\alpha_1)\wedge(\iota_{\pi(q_1)}\alpha_2)\epsilon\otimes
q_2
\end{align*}
Now, equation \eqref{jacobi identity 4} follows from the Leibniz rule.\\
The left hand side of equation \eqref{jacobi identity 5} is equal to
\begin{align*}
& [\alpha_1 \epsilon \otimes q_1,[\alpha_2\otimes
q_2,\alpha_3 \epsilon \otimes q_3]]= \\
& \alpha_1\wedge\iota_{\pi(q_1)}(\alpha_2\wedge L_{\pi(q_2)}\alpha_3)\epsilon\otimes q_3 \pm \alpha_2\wedge(L_{\pi(q_2)}\alpha_3)\wedge\alpha_1\wedge\langle q_1,q_3\rangle \pm \\
& \alpha_{2}\wedge(L_{\pi(q_2)}\alpha_3)\wedge(\iota_{\pi(q_3)}\alpha_1)\epsilon\otimes q_1\pm \alpha_1\wedge\iota_{\pi(q_1)}(\alpha_3\wedge\alpha_2)\epsilon\otimes\{q_2,q_3\}\pm \\
& \alpha_3\wedge\alpha_2\wedge\alpha_1\wedge\langle q_1,\{q_2,q_3\}\rangle \pm \alpha_3\wedge\alpha_2\wedge(\iota_{[\pi(q_1),\pi(q_3)]}\alpha_1)\epsilon\otimes q_1\pm\\
& \alpha_{1}\wedge\iota_{\pi(q_1)}(\alpha_3\wedge\iota_{\pi(q_3)}\alpha_2)\otimes q_2\pm \alpha_3\wedge(\iota_{\pi(q_3)}\alpha_2)\wedge\alpha_1 \wedge d\langle q_1,q_2\rangle\pm \\
&
\alpha_3\wedge(\iota_{\pi(q_3)}\alpha_2)\wedge\alpha_1\epsilon\otimes\{q_1,q_2\}\pm
\alpha_3\wedge(\iota_{\pi(q_3)}\alpha_2)\wedge(L_{\pi(q_2)}\alpha_1)\epsilon\otimes
q_1
\end{align*}
The first summand on the right hand side is equal to
\begin{align*}
& [[\alpha_1 \epsilon \otimes q_1,\alpha_2\otimes
q_2],\alpha_3\epsilon \otimes q_3]= \\
& \alpha_1\wedge(\iota_{\pi(q_1)}\alpha_2)\wedge(L_{\pi(q_2)}\alpha_3)\epsilon\otimes q_3\pm \alpha_3\wedge\alpha_1\wedge(\iota_{\pi(q_1)}\alpha_2)\epsilon\otimes\{q_2,q_3\}\pm\\
& \alpha_3\wedge\iota_{\pi(q_3)}(\alpha_1\wedge\iota_{\pi(q_1)}\alpha_2)\otimes q_2 \pm \alpha_3\wedge\iota_{\pi(q_3)}(\alpha_2\wedge\alpha_1\wedge d\langle q_1,q_2\rangle)\pm\\
& \alpha_2\wedge\alpha_1\wedge(\iota_{[\pi(q_1),\pi(q_2)]}\alpha_3)\epsilon\otimes q_3\pm \alpha_3\wedge\alpha_2\wedge\alpha_1\wedge\langle q_1,\{q_2,q_3\} \rangle \pm\\
& \alpha_3\wedge\iota_{\pi(q_3)}(\alpha_2\wedge\alpha_1)\epsilon\otimes\{q_1,q_2\}\pm \alpha_2\wedge(L_{\pi(q_2)}\alpha_1)\wedge(\iota_{\pi(q_1)}\alpha_3)\epsilon q_3\pm\\
& \alpha_3\wedge\alpha_2\wedge(L_{\pi(q_2)}\alpha_1)\wedge\langle q_1,q_3\rangle \pm \alpha_3\wedge\iota_{\pi(q_3)}(\alpha_2\wedge L_{\pi(q_2)}\alpha_1)\epsilon\otimes q_1
\end{align*}
The second summand on the right hand side is equal to
\begin{align*}
& [\alpha_2
\otimes q_2,[\alpha_1\epsilon\otimes q_1,\alpha_3
\epsilon \otimes q_3]]=\\
& \alpha_2\wedge L_{\pi(q_2)}(\alpha_1\wedge\iota_{\pi(q_1)}\alpha_3)\epsilon\otimes q_3\pm \alpha_1\wedge(\iota_{\pi(q_1)}\alpha_3)\wedge\alpha_2\epsilon\otimes\{q_2,q_3\}\pm\\
& \alpha_1\wedge(\iota_{\pi(q_1)}\alpha_3)\wedge(\iota_{\pi(q_3)}\alpha_2)\otimes q_2\pm \alpha_2\wedge L_{\pi(q_2)}(\alpha_3\wedge\alpha_1\wedge \langle q_1,q_3\rangle)\pm\\
& \alpha_2\wedge L_{\pi(q_2)}(\alpha_3\wedge\iota_{\pi(q_3)}\alpha_1)\epsilon\otimes q_1\pm \alpha_3\wedge(\iota_{\pi(q_3)}\alpha_1)\wedge\alpha_2\epsilon\otimes \{q_1,q_2\}\pm\\
& \alpha_3\wedge(\iota_{\pi(q_3)}\alpha_1)\wedge\alpha_2\epsilon\otimes \pi^{\dagger}(d\langle q_1,q_2\rangle)\pm \alpha_3\wedge(\iota_{\pi(q_3)}\alpha_1)\wedge(\iota_{\pi(q_1)}\alpha_2)\otimes q_2
\end{align*}
Using Leibniz rule the above calculations reduce \eqref{jacobi identity 5} to
\[
0=\alpha_3\wedge(\iota_{\pi(q_3)}\alpha_1)\wedge\alpha_2\wedge d\langle q_1,q_2\rangle-\alpha_3\wedge(\iota_{\pi(q_3)}\alpha_1)\wedge\alpha_2\epsilon\otimes \pi^{\dagger}(d\langle q_1,q_2\rangle)
\]
which is one of the defining relations in $\tau\mathcal{Q}$.\\

In the case of equation \eqref{jacobi identity 6} the left hand side is equal to
\begin{align*}
& [\alpha_1 \otimes q_1,[\alpha_2\otimes q_2,\alpha_3\epsilon\otimes q_3]]=\\
& \alpha_1\wedge L_{\pi(q_1)}(\alpha_2\wedge L_{\pi(q_2)}\alpha_3)\epsilon\otimes q_3\pm \alpha_2\wedge (L_{\pi(q_2)}\alpha_3)\wedge \alpha_1\epsilon\otimes \{q_1,q_3\}\pm \\
& \alpha_2\wedge (L_{\pi(q_2)}\alpha_3)\wedge (\iota_{\pi(q_3)}\alpha_1)\otimes q_1\pm \alpha_1\wedge L_{\pi(q_1)}(\alpha_3\wedge \alpha_2)\epsilon\otimes\{q_2,q_3\}\pm\\
& \alpha_3\wedge \alpha_2\wedge \alpha_1\epsilon\otimes \{q_1,\{q_2,q_3\}\}\pm \alpha_3\wedge \alpha_2\wedge (\iota_{[\pi(q_2),\pi(q_3)]}\alpha_1)\otimes q_1\pm\\
& \alpha_1\wedge L_{\pi(q_1)}(\alpha_3\wedge \iota_{\pi(q_3)}\alpha_2)\otimes q_2\pm \alpha_3\wedge (\iota_{\pi(q_3)}\alpha_2)\wedge \alpha_1\otimes\{q_1,q_2\}\pm\\
& \alpha_3\wedge (\iota_{\pi(q_3)}\alpha_2)\wedge (L_{\pi(q_2)}\alpha_1)\otimes q_1
\end{align*}
The first summand of the right hand side is equal to
\begin{align*}
& [[\alpha_1\otimes q_1,\alpha_2\otimes q_2],\alpha_3\epsilon\otimes q_3]=\\
&\alpha_1\wedge (L_{\pi(q_1)}\alpha_2)\wedge (L_{\pi(q_2)}\alpha_3)\epsilon\otimes q_3\pm \alpha_3\wedge \alpha_1\wedge (L_{\pi(q_1)}\alpha_2)\epsilon\otimes\{q_2,q_3\}\pm\\
& \alpha_3\wedge \iota_{\pi(q_3)}(\alpha_1\wedge L_{\pi(q_1)}\alpha_2)\otimes q_2\pm \alpha_2\wedge \alpha_1\wedge (L_{[\pi(q_1),\pi(q_2)]}\alpha_3)\epsilon\otimes q_3 \pm \\
& \alpha_3\wedge \alpha_2\wedge \alpha_1\epsilon\otimes\{\{q_1,q_2\},q_3\}\pm \alpha_3\wedge \iota_{\pi(q_3)}(\alpha_2\wedge \alpha_1)\otimes\{q_1,q_2\}\pm\\
& \alpha_2\wedge (L_{\pi(q_2)}\alpha_1)\wedge (L_{\pi(q_1)}\alpha_3)\epsilon\otimes q_3\pm \alpha_3\wedge \alpha_2\wedge (L_{\pi(q_2)}\alpha_1)\otimes\{q_1,q_3\}\pm \\
& \alpha_3\wedge \iota_{\pi(q_3)}(\alpha_2\wedge L_{\pi(q_2)}\alpha_1)\otimes q_1
\end{align*}
The second summand is equal to
\begin{align*}
& [\alpha_2\otimes q_2,[\alpha_1\otimes q_1,\alpha_3\epsilon\otimes q_3]]=\\
& \alpha_2\wedge L_{\pi(q_2)}(\alpha_1\wedge L_{\pi(q_1)}\alpha_3)\epsilon\otimes q_3\pm \alpha_1\wedge (L_{\pi(q_1)}\alpha_3)\wedge \alpha_2\epsilon\otimes \{q_2,q_3\}\pm \\
& \alpha_1\wedge (L_{\pi(q_1)}\alpha_3)\wedge (\iota_{\pi(q_3)}\alpha_2)\otimes q_2\pm \alpha_2\wedge L_{\pi(q_2)}(\alpha_3\wedge \alpha_1)\epsilon\otimes\{q_1,q_3\}\pm\\
& \alpha_3\wedge \alpha_1\wedge \alpha_2\epsilon\otimes \{q_2,\{q_1,q_3\}\}\pm \alpha_3\wedge \alpha_1\wedge (\iota_{[\pi(q_1),\pi(q_3)]}\alpha_2)\otimes q_2\pm\\
& \alpha_2\wedge L_{\pi(q_2)}(\alpha_3\wedge \iota_{\pi(q_3)}\alpha_1)\otimes q_1\pm \alpha_3\wedge (\iota_{\pi(q_3)}\alpha_1)\wedge \alpha_2\otimes\{q_1,q_2\}\pm\\
& \alpha_3\wedge (\iota_{\pi(q_3)}\alpha_1)\wedge \alpha_2\otimes\pi^{\dagger}(d\langle q_1,q_2\rangle) \pm \alpha_3\wedge (\iota_{\pi(q_3)}\alpha_1)\wedge (L_{\pi(q_1)}\alpha_2)\otimes q_2
\end{align*}
Using the Leibniz rule the above calculations reduce \eqref{jacobi identity 6} to
\[
0=\alpha_3\wedge(\iota_{\pi(q_3)}\alpha_1)\wedge\alpha_2\otimes\pi^{\dagger}(d\langle q_1,q_2\rangle)
\]
which holds since $\alpha_3\wedge(\iota_{\pi(q_3)}\alpha_1)\wedge\alpha_2\otimes\pi^{\dagger}(d\langle q_1,q_2\rangle)=\alpha_3\wedge(\iota_{\pi(q_3)}\alpha_1)\wedge\alpha_2(dd\langle q_1,q_2\rangle)$.\\
The left hand side of equation \eqref{jacobi identity 7} is equal to:
\begin{align*}
& [\alpha_1 \otimes q_1,[\alpha_2\epsilon\otimes q_2,\alpha_3\otimes q_3]]=\\
& \alpha_1\wedge L_{\pi(q_1)}(\alpha_2\wedge\iota_{\pi(q_2)}\alpha_3)\otimes q_3 \pm \alpha_2\wedge(\iota_{\pi(q_2)}\alpha_3)\wedge\alpha_1\otimes\{q_1,q_3\}\pm\\
& \alpha_2\wedge(\iota_{\pi(q_2)}\alpha_3)\wedge(L_{\pi(q_3)}\alpha_1)\otimes q_1\pm \alpha_1\wedge L_{\pi(q_1)}(\alpha_3\wedge \alpha_2\wedge d\langle q_2,q_3\rangle)\pm\\
& \alpha_1\wedge L_{\pi(q_1)}(\alpha_3\wedge\alpha_2)\epsilon\otimes \{q_2,q_3\}\pm \alpha_3\wedge\alpha_2\wedge\alpha_1\epsilon\otimes\{q_1,\{q_2,q_3\}\}\pm\\
& \alpha_3\wedge\alpha_2\wedge(\iota_{[\pi(q_2),\pi(q_3)]}\alpha_1)\otimes q_1\pm \alpha_1\wedge L_{\pi(q_1)}(\alpha_3\wedge L_{\pi(q_3)}\alpha_2)\epsilon\otimes q_2 \pm \\
& \alpha_3\wedge(L_{\pi(q_3)}\alpha_2)\wedge\alpha_1\epsilon\otimes[q_1,q_2]\pm \alpha_3\wedge(L_{\pi(q_3)}\alpha_2)\wedge(\iota_{\pi(q_2)}\alpha_1)\otimes q_1
\end{align*}
The first summand of the right hand side is equal to:
\begin{align*}
& [[\alpha_1\otimes q_1,\alpha_2\epsilon\otimes q_2],\alpha_3\otimes q_3]=\\
& \alpha_1\wedge(L_{\pi(q_1)}\alpha_2)\wedge(\iota_{\pi(q_2)}\alpha_3)\otimes q_3\pm \alpha_3\wedge\alpha_1\wedge(L_{\pi(q_1)}\alpha_2)\wedge d\langle q_2,q_3\rangle\pm\\
& \alpha_3\wedge\alpha_1\wedge(L_{\pi(q_1)}\alpha_2)\epsilon\otimes[q_2,q_3]\pm \alpha_3\wedge L_{\pi(q_3)}(\alpha_1\wedge L_{\pi(q_1)}\alpha_2)\epsilon\otimes q_2\pm\\
& \alpha_2\wedge\alpha_1\wedge(\iota_{[\pi(q_1),\pi(q_2)]}\alpha_3)\otimes q_3\pm \alpha_3\wedge\alpha_2\wedge\alpha_1\wedge d\langle [q_1,q_2],q_3\rangle\pm\\
& \alpha_3\wedge\alpha_2\wedge\alpha_1\epsilon\otimes\{\{q_1,q_2\},q_3\}\pm \alpha_3\wedge L_{\pi(q_3)}(\alpha_2\wedge\alpha_1)\epsilon\otimes \{q_1,q_2\}\pm\\
& \alpha_2\wedge(\iota_{\pi(q_2)}\alpha_1)\wedge L_{\pi(q_1)}\alpha_3\otimes q_3\pm \alpha_3\wedge\alpha_2\wedge(\iota_{\pi(q_2)}\alpha_1)\otimes\{q_1,q_3\}\pm\\
& \alpha_3\wedge L_{\pi(q_3)}(\alpha_2\wedge\iota_{\pi(2)}\alpha_1)\otimes q_1
\end{align*}
The second summand is equal to
\begin{align*}
& [\alpha_2\epsilon\otimes q_2,[\alpha_1\otimes q_1,\alpha_3\otimes q_3]]=\\
& \alpha_2\wedge\iota_{\pi(q_2)}(\alpha_1\wedge L_{\pi(q_1)}\alpha_3)\otimes q_3\pm \alpha_1\wedge(L_{\pi(q_1)}\alpha_3)\wedge\alpha_2\wedge d\langle q_2,q_3\rangle\pm\\
& \alpha_1\wedge(L_{\pi(q_1)}\alpha_3)\wedge\alpha_2\epsilon\otimes\{q_2,q_3\}\pm \alpha_1\wedge(L_{\pi(q_1)}\alpha_3)\wedge L_{\pi(q_3)}\alpha_2\epsilon\otimes q_2\\
& \alpha_2\wedge\iota_{\pi(q_2)}(\alpha_3\alpha_1)\otimes \{q_1,q_3\}\pm \alpha_3\wedge\alpha_1\wedge\alpha_2\wedge d\langle q_2,\{q_1,q_3\}\rangle\pm\\
& \alpha_3\wedge\alpha_1\wedge\alpha_2\epsilon\otimes\{q_2,\{q_1,q_3\}\}\pm \alpha_3\wedge\alpha_1\wedge(L_{[\pi(q_1),\pi(q_3)]}\alpha_2)\epsilon\otimes q_2\pm\\
& \alpha_2\wedge\iota_{\pi(q_2)}(\alpha_3\wedge L_{\pi(q_3)}\alpha_1)\otimes q_1\pm \alpha_3\wedge(L_{\pi(q_3)}\alpha_1)\wedge\alpha_2\wedge d\langle q_1,q_2\rangle\pm\\
& \alpha_3\wedge(L_{\pi(q_3)}\alpha_1)\wedge\alpha_2\epsilon\otimes\{q_1,q_2\}\pm \alpha_3\wedge(L_{\pi(q_3)}\alpha_1)\wedge\alpha_2\wedge d\langle q_1,q_2\rangle \pm \\
& \alpha_3\wedge(L_{\pi(q_3)}\alpha_1)\wedge(L_{\pi(q_1)}\alpha_2)\epsilon\otimes q_2
\end{align*} Now, equation \eqref{jacobi identity 7} follows from the Leibniz rule.
It is left to the reader verification of equation \eqref{jacobi
identity 8} using Leibniz rule.

\subsubsection{Leibniz rule} Suppose that $a=\omega_1+\alpha_1\epsilon\otimes
q_1+\beta_1\otimes r_1$ and
$b=\omega_2+\alpha_2\epsilon\otimes q_2+\beta_2\otimes
r_2$ are homogeneous elements, where
$\omega_i,\alpha_i,\beta_i\in\mathcal{O}_{X^{\sharp}}$ and
$q_i,r_i\in\mathcal{Q}$ for $i=1,2$. For a homogeneous $\psi\in\mathcal{O}_{X^{\sharp}}$, the Leibniz rule says
\begin{align}\label{Leibniz in tau Q}
[a,\psi\cdot b]=\sigma(a)(\psi)\cdot
b+(-1)^{a\psi}\psi\cdot[a,b].
\end{align}
We verify \eqref{Leibniz in tau Q} in the case \eqref{bracket with
scalars-1} leaving the cases \eqref{bracket with
scalars-2}-\eqref{bracket with scalars-4} to the reader.
\begin{align*}
[\alpha_1&\epsilon\otimes q_1,\gamma\cdot(\alpha_2\epsilon\otimes q_2)]=[\alpha_1\epsilon\otimes q_1,(\gamma\wedge\alpha_2)\epsilon\otimes q_2]=\\
& =\alpha_1\wedge\iota_{\pi(q_1)}(\gamma\wedge\alpha_2)\epsilon\otimes q_2\pm \gamma\wedge\alpha_2\wedge\alpha_1\wedge\langle q_1,q_2\rangle \pm \gamma\wedge\alpha_2\wedge\iota_{\pi(q_2)}\alpha_1\epsilon\otimes q_1\\
& =\gamma\cdot[\alpha_1\epsilon\otimes q_1,\alpha_2\epsilon\otimes q_2] \pm \alpha_1\wedge(\iota_{\pi(q_1)}\gamma)\wedge(\alpha_2\epsilon\otimes q_2)\\
& = \gamma\cdot[\alpha_1\epsilon\otimes q_1,\alpha_2\epsilon\otimes q_2] \pm (\pi(\alpha_1\epsilon\otimes q_1)\gamma)\cdot(\alpha_2\epsilon\otimes q_2).
\end{align*}

\end{document}